\documentclass{siamart1116}

\usepackage{amsopn}
\usepackage{amsmath}
\usepackage{amssymb}
\usepackage{enumerate}
\usepackage{bm}
\usepackage{booktabs}
\usepackage{algorithm,algorithmic}
\usepackage{datetime}
\usepackage[protrusion=true,tracking=false,kerning=true]{microtype}
\usepackage{pdfsync}
\usepackage{color}
\usepackage{multirow}

\newcommand{\deletethis}[1]{{}}

\numberwithin{theorem}{section}

\theoremstyle{remark}
\newtheorem{remark}[theorem]{Remark}

\newcommand{\TheTitle}{Sum-of-Squares Optimization Without Semidefinite Programming}
\newcommand{\TheAuthors}{D{\'a}vid PAPP and Sercan YILDIZ}

\headers{\TheTitle}{\TheAuthors}

\title{{\TheTitle}\thanks{Original manuscript: December 5, 2017. Revised on Thursday June 14, 2018 and on \today. \textbf{This is a technical report of the same title as the manuscript accepted for publication in the SIAM Journal on Optimization. Due to the journal's page limit, Section 7.3 of this report is omitted in the journal version.}		
\funding{This material is based upon work supported by the National Science Foundation under Grant No.~DMS-1719828. Additionally, this material was based upon work partially supported by the National Science Foundation under Grant No.~DMS-1638521 to the Statistical and Applied Mathematical Sciences Institute. Any opinions, findings, and conclusions or recommendations expressed in this material are those of the authors and do not necessarily reflect the views of the National Science Foundation.}}
}

\author{
  D{\'a}vid PAPP\thanks{North Carolina State University, Department of Mathematics. Email: \email{dpapp@ncsu.edu}.}
  \and
  Sercan YILDIZ\thanks{University of North Carolina at Chapel Hill, Department of Statistics and Operations Research; Statistical and Applied Mathematical Sciences Institute.}
}

\ifpdf
\hypersetup{
  pdftitle={\TheTitle},
  pdfauthor={\TheAuthors}
}
\fi

\DeclareMathOperator*{\minimize}{minimize}
\DeclareMathOperator*{\maximize}{maximize}
\DeclareMathOperator*{\st}{subject\;to}
\DeclareMathOperator*{\cond}{cond}

\DeclareMathOperator*{\diag}{diag}

\DeclareMathOperator*{\tr}{tr}
\let\vec\relax \DeclareMathOperator*{\vec}{vec}

\newcommand{\defeq}{\ensuremath{\overset{\mathrm{def}}{=}}}

\newcommand{\vA}{\mathbf{A}}

\newcommand{\vb}{\mathbf{b}}
\newcommand{\vB}{\mathbf{B}}
\newcommand{\vc}{\mathbf{c}}
\newcommand{\vC}{\mathbf{C}}
\newcommand{\vd}{\mathbf{d}}

\newcommand{\ve}{\mathbf{e}}
\newcommand{\vE}{\mathbf{E}}

\newcommand{\vg}{\mathbf{g}}
\newcommand{\vG}{\mathbf{G}}

\newcommand{\vI}{\mathbf{I}}

\newcommand{\vell}{\bm{\ell}}
\newcommand{\vL}{\mathbf{L}}
\newcommand{\vLambda}{\bm{\Lambda}}

\newcommand{\vp}{\mathbf{p}}
\newcommand{\vq}{\mathbf{q}}
\newcommand{\vP}{\mathbf{P}}
\newcommand{\vQ}{\mathbf{Q}}
\newcommand{\vr}{\mathbf{r}}
\newcommand{\vR}{\mathbf{R}}
\newcommand{\vs}{\mathbf{s}}
\newcommand{\vS}{\mathbf{S}}
\newcommand{\vt}{\mathbf{t}}
\newcommand{\vu}{\mathbf{u}}
\newcommand{\vv}{\mathbf{v}}
\newcommand{\vV}{\mathbf{V}}
\newcommand{\vw}{\mathbf{w}}

\newcommand{\vx}{\mathbf{x}}
\newcommand{\vy}{\mathbf{y}}
\newcommand{\vX}{\mathbf{X}}

\newcommand{\vz}{\mathbf{z}}

\newcommand{\vzero}{\mathbf{0}}
\newcommand{\vone}{\mathbf{1}}

\newcommand{\R}{\mathbb{R}}

\newcommand{\real}{\mathbb{R}}
\newcommand{\bS}{\mathbb{S}}

\newcommand{\T}{\mathrm{T}}

\newcommand{\cB}{\mathcal{B}}
\newcommand{\cC}{\mathcal{C}}
\newcommand{\cD}{\mathcal{D}}

\newcommand{\cF}{\mathcal{F}}

\newcommand{\cK}{\mathcal{K}}

\newcommand{\cN}{\mathcal{N}}
\newcommand{\Oh}{\mathcal{O}}
\newcommand{\cP}{\mathcal{P}}

\newcommand{\cS}{\mathcal{S}}
\newcommand{\cT}{\mathcal{T}}
\newcommand{\cV}{\mathcal{V}}

\newcommand{\dop}{\mathrm{d}}
\newcommand{\lambdamin}{\lambda_\textup{min}}
\newcommand{\lambdamax}{\lambda_\textup{max}}

\newcommand{\af}{\alpha}

\newcommand{\bx}{\bar{\vx}}

\newcommand{\bs}{\bar{\vs}}

\newcommand{\bF}{\bar{F}}
\newcommand{\bg}{\bar{g}}
\newcommand{\bH}{\bar{H}}
\newcommand{\bK}{\bar{\cK}}
\newcommand{\bnu}{\bar{\nu}}

\newcommand{\dx}{\Delta_{\bx}}
\newcommand{\dy}{\Delta_{\vy}}
\newcommand{\ds}{\Delta_{\bs}}
\newcommand{\dz}{\Delta_{\vz}}

\newcommand{\PPTsq}{\ensuremath{(\vP\vP^\T)^{\circ2}}}

\newcommand{\revise}[1]{{#1}}

\begin{document}

\maketitle

\begin{abstract}
We propose a homogeneous primal-dual interior-point method to solve sum-of-squares optimization problems by combining non-symmetric conic optimization techniques and polynomial interpolation. The approach optimizes directly over the sum-of-squares cone and its dual, circumventing the semidefinite programming (SDP) reformulation which requires a large number of auxiliary variables \revise{when the degree of sum-of-squares polynomials is large}. As a result, it has substantially lower theoretical time and space complexity than the conventional SDP-based approach. Although our approach avoids the semidefinite programming reformulation, an optimal solution to the semidefinite program can be recovered with little additional effort. Computational results confirm that \revise{the proposed method is several orders of magnitude faster than the SDP-based approach for optimization problems over high-degree sum-of-squares polynomials}.
\end{abstract}



\section{Introduction}\label{sec:introduction}

We propose a homogeneous primal-dual interior-point algorithm for sum-of-squares optimization. Our approach is applicable to optimization problems over products of sum-of-squares cones, which include the optimization of polynomials over basic semialgebraic sets, moment problems, and parametric sum-of-squares problems. 
These problems are fundamental in many areas of applied mathematics and engineering, including discrete geometry 
\cite{BachocVallentin2008,BallingerBlekhermanCohnGiansiracusaKellySchurmann2009}, 
probability theory \cite{BertsimasPopescu2005}, control theory 
\cite{HessHenrionLasserrePham2016}, 
signal processing \cite{Dumitrescu2017},
power systems engineering 
\cite{GhaddarMarecekMevissen2016}, computational algebraic geometry 
\cite{Laurent2009,KaltofenLiYangZhi2008},
design of experiments 
\cite{Papp2012}, and statistical estimation \cite{AlizadehPapp2013}.
Additional applications of sum-of-squares optimization are described in \cite{BlekhermanParriloThomas2013}.

In the simplest form of polynomial optimization, we are given $n$-variate polynomials $g_1,\ldots,g_m$ and $f$ over the reals, and we are interested in determining the minimum value of $f$ on the basic closed semialgebraic set
\begin{equation}\label{eq:domain}
\cS\defeq\left\{\vt\in\R^n\,|\, g_i(\vt)\geq 0\quad\forall\,i=1,\ldots,m\right\}.
\end{equation}
That is to say, we would like to compute
\begin{equation}\label{eq:poly-opt}
\inf_{\vt\in\R^n} \left\{f(\vt)\,|\,\vt\in\cS\right\}.
\end{equation}
Equivalently, one may seek the largest constant $c\in\R$ which can be subtracted from $f$ such that $f-c$ is nonnegative on the set $\cS$. Thus, the polynomial optimization problem \eqref{eq:poly-opt} can be reduced to the problem of checking polynomial nonnegativity.

In many of the applications mentioned above, the goal is not to simply compute the minimum value of a \emph{given} polynomial, but rather to \emph{find} an optimal polynomial satisfying certain \emph{shape constraints} that impose bounds on certain linear functionals of the polynomial. In this setting, even the case of polynomials with only a few variables is of great interest; in fact, most of the references cited above are concerned with univariate polynomials of high degree.	
	

These problems are most naturally formulated as conic optimization problems: Let $\cK\subset\R^N$ be a closed and convex cone. A \emph{conic optimization problem} is a problem of the form
\begin{equation}\label{eq:CP-P}
\begin{aligned}
&\minimize_{\vx\in\R^N}\quad    && \vc^\T\vx\\
&\st\;\;                        && \vA\vx = \vb\\
&&& \vx\in\cK
\end{aligned}
\end{equation}
where $\vA$ is a $k\times N$ real matrix, and $\vc$ and $\vb$ are real vectors of appropriate dimensions. Its dual problem is
\begin{equation}\label{eq:CP-D}
\begin{aligned}
&\maximize_{\vy\in\R^k,\,\vs\in\R^N}\quad   && \vb^\T\vy\\
&\st\;\;                                    && \vA^\T\vy + \vs = \vc\\
&&& \vs\in\cK^*.
\end{aligned}
\end{equation}
Here $\cK^*\defeq\{\vs\in\R^N\,|\,\vs^\T\vx\geq 0\;\forall\,\vx\in\cK\}$ denotes the dual cone of $\cK$, which is also closed and convex. In the literature on conic optimization, it is usually assumed that the cone $\cK$ is pointed and has nonempty interior in addition to being closed and convex; in this case, $\cK$ is called a \emph{proper} cone. When $\cK$ is proper, its dual cone $\cK^*$ is also proper.

In optimization problems involving polynomials, we are typically interested in the space of $n$-variate polynomials of total degree at most $r$, which we denote with $\cV_{n,r}$ in this paper, and the closed convex cone $\cP_{n,r}^{\cS}$ of polynomials that are nonnegative on $\cS$:
\[ \cP_{n,r}^{\cS} \defeq \left\{ p\in\cV_{n,r}\,|\, p(\vt)\geq 0 \quad\forall\, \vt\in\cS \right\}. \]
In the case $\cS=\R^n$, we use the lighter notation $\cP_{n,r}$ to represent the cone of polynomials that are nonnegative everywhere. The dual cone of $\cP_{n,r}^{\cS}$ is known as the \emph{moment cone} corresponding to $\cS$.

Throughout the paper, all polynomials are $n$-variate polynomials over the real number field. The degree of a polynomial is always understood in the sense of total degree, and all vectors are interpreted as column vectors unless stated otherwise. We represent vectors and matrices in boldface type to distinguish them from scalars. We let $\vzero$ and $\vone$ denote the all-zeros and all-ones vectors, and we let $\ve_i$ denote the $i$-th standard unit vector whose only nonzero entry is at the $i$-th position and equal to 1. We represent the arguments of an $n$-variate polynomial with $\vt=(t_1,\ldots,t_n)$ when necessary. We let $\cC^\circ$ denote the interior of a set $\cC\subset\R^N$.

\subsection{Sum-of-squares polynomials: basic definitions and notation}

A polynomial $p\in\cV_{n,2d}$ is said to be \emph{sum-of-squares} (SOS) if it can be expressed as a finite sum of squared polynomials. More precisely, the polynomial $p\in\cV_{n,2d}$ is SOS if there exist $q_1,\ldots,q_M\in\cV_{n,d}$ such that $p = \sum_{j=1}^M q_j^2$. We let $\Sigma_{n,2d}$ denote the set consisting of $n$-variate SOS polynomials of degree $2d$. This set is a proper cone in $\cV_{n,2d}$ \cite[Thm. 17.1]{Nesterov2000}. Let $\vg\defeq(g_1,\ldots,g_m)$ and $\vd\defeq(d_1,\ldots,d_m)$ for some given nonzero polynomials $g_1,\ldots,g_m$ and nonnegative integers $d_1,\ldots,d_m$. Consider the space $\cV_{n,2\vd}^\vg$ of polynomials $p$ for which there exist $r_1\in\cV_{n,2d_1},\ldots,r_m\in\cV_{n,2d_m}$ such that $p = \sum_{i=1}^m g_i r_i$.
A polynomial $p\in\cV_{n,2\vd}^\vg$ is said to be \emph{weighted sum-of-squares} (WSOS) if there exist $s_1\in\Sigma_{n,2d_1},\ldots,s_m\in\Sigma_{n,2d_m}$ such that $p = \sum_{i=1}^m g_i s_i$. We let $\Sigma_{n,2\vd}^\vg$ denote the set consisting of these WSOS polynomials. This set is a convex cone with nonempty interior in $\cV_{n,2\vd}^\vg$, but it is not always closed or pointed. Proposition~\ref{thm:WSOSprops} below characterizes when $\Sigma_{n,2\vd}^\vg$ is a proper cone.
An \emph{SOS optimization problem} is a conic optimization problem where the underlying cone is a Cartesian product of SOS and WSOS cones. For simplicity, we limit our initial presentation in this paper to optimization problems over SOS cones, and discuss the more general case of optimization over WSOS cones in Section~\ref{sec:WSOS}.

Let $L\defeq\dim\cV_{n,d} = {n+d \choose n}$ and $U\defeq\dim\cV_{n,2d}={n+2d \choose n}$ denote the dimensions of the spaces of $n$-variate polynomials of degree at most $d$ and $2d$, respectively.
The space $\cV_{n,2d}$ is isomorphic to $\R^U$; therefore, $\Sigma_{n,2d}$ can equivalently be seen as a cone in $\R^U$. In view of this connection, given an ordered basis $\vq = (q_1,\ldots,q_U)$ of $\cV_{n,2d}$, we say that a vector $\vs=(s_1,\ldots,s_U)\in\R^U$ satisfies $\vs\in\Sigma_{n,2d}$ if the polynomial $\sum_{u=1}^Us_uq_u$ is SOS.
We let $\bS^L$ denote the space of $L\times L$ real symmetric matrices and let $\bS^L_+$ (resp. $\bS^L_{++}$) denote the cone of positive semidefinite (resp. positive definite) matrices in the same space. When the size of the matrices is clear from the context, we write $\vX\succcurlyeq\vzero$ (resp. $\vX\succ\vzero$) to mean that the real symmetric matrix $\vX$ is positive semidefinite (resp. positive definite). For matrices $\vS,\vX\in\bS^L$, the notation $\vS\bullet\vX=\sum_{i=1}^L\sum_{j=1}^L\vS_{ij}\vX_{ij}$ represents the Frobenius inner product of $\vS$ and $\vX$, and $\|\vX\|_F=\sqrt{\vX\bullet\vX}$ represents the Frobenius norm of $\vX$.

An SOS decomposition provides a simple certificate demonstrating the global nonnegativity of a polynomial. The key observation behind modern polynomial optimization approaches is that while deciding whether a polynomial is nonnegative is NP-hard (outside of a few very special cases), the cone of SOS polynomials admits a semidefinite representation \cite{ParriloThesis2000,Nesterov2000,Lasserre2001,Laurent2009}. Using this representation, optimization problems over SOS cones can be reformulated as semidefinite programming (SDP) problems. The following theorem is due to Nesterov \cite{Nesterov2000}; we present it here in our notation for completeness.

\begin{proposition}[{\cite[Thm. 17.1]{Nesterov2000}}]\label{thm:Nesterov2000}
Fix ordered bases $\vp = (p_1,\ldots,p_L)$ and $\vq = (q_1,\ldots,q_U)$ of $\cV_{n,d}$ and $\cV_{n,2d}$, respectively. Let $\Lambda:\R^U\to\bS^L$ be the unique linear mapping satisfying $\Lambda(\vq)=\vp\vp^\T$, and let $\Lambda^*$ denote its adjoint. Then $\vs\in\Sigma_{n,2d}$ \deletethis{the polynomial $\sum_{u=1}^Us_uq_u$ belongs to the cone $\Sigma_{n,2d}$} if and only if there exists a matrix $\vS\succcurlyeq\vzero$ satisfying
\begin{equation*}\label{eq:Sigma}
\vs = \Lambda^*(\vS).
\end{equation*}
Additionally, the dual cone of $\Sigma_{n,2d}$ admits the characterization
\begin{equation*}\label{eq:Sigma*}
\Sigma_{n,2d}^* = \left\{\vx\in\real^U\,|\,\Lambda(\vx)\succcurlyeq\vzero\right\}.
\end{equation*}
\end{proposition}

We emphasize that the operator $\Lambda$ in Proposition~\ref{thm:Nesterov2000} depends explicitly on the specific bases $\vp$ and $\vq$ chosen to represent $\cV_{n,d}$ and $\cV_{n,2d}$. In particular, these bases determine which linear slice of the positive semidefinite cone is used to characterize $\Sigma_{n,2d}^*$. For instance, a standard choice in practice is to choose $\vp$ and $\vq$ as monomials up to degree $d$ and $2d$ respectively; with this choice and $n=1$, the operator $\Lambda$ becomes the mapping from $\vx$ to its Hankel matrix $(x_{i+j})_{i,j=0,\ldots,d}$. We shall explore the impact of the choice of $\vp$ and $\vq$ in more detail in Section~\ref{sec:barriers}.

Nesterov \cite{Nesterov2000} also gave an analogous semidefinite representation for WSOS cones. We postpone the precise statement of this result to Proposition~\ref{thm:Nesterov2000WSOS} below and mention here only that with $m$ polynomial weights, the semidefinite representation of $\Sigma_{n,2\vd}^\vg$ requires $m$ positive semidefinite matrices of orders ${n+d_1\choose n},\ldots,{n+d_m\choose n}$ respectively. In the characterization of the dual cone, this translates to $m$ linear matrix inequalities.

Now one can invoke a \emph{Positivstellensatz} result such as those of Putinar \cite{Putinar1993}, Schm{\"u}dgen \cite{Schmudgen1991}, Handelman \cite{Handelman1988}, or P{\'o}lya \cite[p.57]{HardyLittlewoodPolya1934} to conclude that under certain conditions on the polynomials $g_1,\ldots,g_m$, every polynomial that is nonnegative (or strictly positive) on $\cS$ is WSOS with respect to polynomial weights that are constructed from $g_1,\ldots,g_m$ and that are trivially nonnegative on $\cS$. Once the weights and the degrees of the polynomials to be squared have been fixed, the resulting WSOS cone is an inner approximation of $\cP_{n,r}^\cS$. This leads to a hierarchy of semidefinite programs parameterized with increasing degrees of squared polynomials (and possibly increasingly larger sets of weights). For optimization problems over cones of nonnegative polynomials, each level of this hierarchy provides increasingly better primal bounds, and the convergence of these bounds to the optimal value is guaranteed by the \emph{Positivstellensatz} result invoked.

These connections between nonnegative polynomials, sums of squares, and semidefinite programming were first made in \cite{Nesterov2000,ParriloThesis2000,Lasserre2001}; see \cite{Laurent2009} for a comprehensive review. The textbooks \cite{Marshall2008,BlekhermanParriloThomas2013,Lasserre2015} provide an excellent introduction to these techniques and highlight their connections to different areas of pure and applied mathematics.
The review \cite{deKlerk2008} places these results in the context of complexity theory with an overview of hardness theorems and approximation schemes for polynomial optimization on standard domains.

There are several implementations available (mostly in the form of Matlab packages) for the numerical solution of polynomial optimization problems using SOS theory. These include SOSTOOLS \cite{sostools}, GloptiPoly \cite{gloptipoly}, SparsePOP \cite{sparsepop}, SPOT \cite{spot}, GpoSolver \cite{gposolver}, SOSOPT \cite{sosopt}. All of these implementations rely on the SDP-based approach outlined above.

\subsection{Complexity of the semidefinite representation}\label{sec:complexity}

While Proposition~\ref{thm:Nesterov2000} shows that the SOS cone and its dual are semidefinite representable, this representation is rather inefficient. The cone $\Sigma_{n,2d}^*$ is characterized as a $U$-dimensional linear slice of the cone of $L\times L$ positive semidefinite matrices, and similarly $\Sigma_{n,2d}$ is represented as a linear image of the same positive semidefinite cone. The computational implication is that $\Theta(L^2)$ variables are needed to represent a $U$-dimensional vector belonging to $\Sigma_{n,2d}$. This leads to a substantial increase in the number of variables when dealing with optimization problems over these cones. For instance, in the case $n=1$, the parameters $L$ and $U$ are equal to $d+1$ and $2d+1$ respectively, and while the cone $\Sigma_{1,2d}$ has dimension $2d+1$, its semidefinite representation requires $d(d+1)/2$ variables. This has a significant effect on the time complexity of optimization over $\Sigma_{n,2d}$ and its dual. Standard primal-dual interior-point methods for semidefinite programming have an $\Oh(L^{6.5})$ running time for problems with a single $L\times L$ matrix variable, as their iteration complexity is $\Oh(L^{0.5})$ (see, e.g., \cite{MonteiroTodd2000}) and each iteration requires the solution of a linear system in $\Theta(L^2)$ variables (which costs $\Oh(L^6)$ arithmetic operations).

The situation is even worse for optimization problems over WSOS cones: The size of the semidefinite representation of the WSOS cone $\Sigma_{n,2\vd}^\vg$ grows linearly with the number $m$ of polynomial weights, whereas its intrinsic dimension remains the same. In particular, assuming for simplicity that all $d_i$'s have the same value $d$, the semidefinite representation of $\Sigma_{n,2\vd}^\vg$ requires $m$ matrix variables of order $L={n+d\choose n}$ each, which implies an $\Oh(m^{0.5} L^{0.5})$ iteration complexity \cite{MonteiroTodd2000} and $\Oh(m^{1.5} L^{6.5})$ running time for optimization over $\Sigma_{n,2\vd}^\vg$ using standard primal-dual interior-point methods for semidefinite programming.

The impracticality of the SDP-based approach for problems involving high-degree polynomials was demonstrated recently in \cite{Papp2017}. In a family of small examples with two WSOS cone constraints (for $n=1$, $m=2$, and increasing $d$), the largest instance that could be solved with the 32GB of available memory using several popular semidefinite programming solvers had $2d=1100$, even though the original SOS optimization problem (before the semidefinite programming reformulation) is a convex optimization problem with only $2d+1$ variables. The increase in the running times as $d$ increases was also found prohibitive in practice, with the largest ``solvable'' instances requiring several hours of computation.

\subsection{Contributions and outline of the paper}

We describe a primal-dual interior-point method for SOS optimization that circumvents the inefficiencies of the semidefinite programming formulation. Our approach adapts a recent algorithm of Skajaa and Ye \cite{SkajaaYe2015} for non-symmetric conic optimization to optimize directly over WSOS cones. We describe this algorithm in Section~\ref{sec:SY}. We provide a brief review of the necessary background on barrier functions in Appendix~\ref{sec:LHSCB} to make the paper self-contained.

Following \cite{LofbergParrilo2004}, our approach takes advantage of the interpolant basis representation of polynomials for fast and stable computation of Newton steps inside the algorithm.
We discuss the complexity of computing the Newton step in the monomial, Chebyshev, and interpolant basis representations in Section~\ref{sec:barriers}. We study how the numerical conditioning of the interpolant basis representation depends on the chosen interpolation points and the basis of the space of polynomials to be squared in Section~\ref{sec:stability}. We also compare the conditioning of this representation against the conditioning of the monomial and Chebyshev representations in this section.

For optimization problems over the cone $\Sigma_{n,2d}$, our approach leads to an algorithm with $\Oh(L^{0.5})$ iteration complexity and $\Oh(L^{0.5}U^3)$ running time. This compares favorably against the $\Oh(L^{6.5})$ time required for the solution of the corresponding semidefinite programs \revise{especially for problems with large $d$}. 
For instance, in the case $n=1$ our method has $\Oh(d^{3.5})$ running time, whereas the standard SDP-based approach requires $\Oh(d^{6.5})$ time.	

Although our approach circumvents the solution of the conventional semidefinite programming formulation of an SOS optimization problem, we show in Section~\ref{sec:recovery} that an optimal solution to the semidefinite program can be recovered with little additional effort. This is necessary, for example, to construct explicit certificates proving that the optimal SOS polynomials computed using our approach are indeed SOS \cite{KaltofenLiYangZhi2008,PeyrlParrilo2008}.

We generalize the results of the previous sections to optimization problems over WSOS cones in Section~\ref{sec:WSOS} and present the results of our numerical experiments in Section~\ref{sec:experiments}. \revise{The results demonstrate that the proposed approach can have significant practical advantages over the standard SDP-based approach for problems requiring high-degree SOS polynomials.}

\deletethis{
We describe a primal-dual interior-point method for SOS optimization that circumvents the inefficiencies of the semidefinite programming formulation. Our approach adapts a recent algorithm of Skajaa and Ye \cite{SkajaaYe2015} for non-symmetric conic optimization to optimize directly over WSOS cones. The iteration complexity of this algorithm is the same as that of the semidefinite programming approach.

Following \cite{LofbergParrilo2004}, our approach takes advantage of the interpolant basis representation of polynomials for fast and stable computation of Newton steps inside the algorithm. We study how the conditioning of the problem is affected by the choice of bases and interpolation points and compare our approach against using the monomial and Chebyshev bases.

For optimization problems over the cone $\Sigma_{n,2d}$, this leads to an algorithm with $\Oh(L^{0.5}U^3)$ running time, which compares favorably against the $\Oh(L^{6.5})$ time required for the solution of the corresponding semidefinite programs. In the case $n=1$, we obtain an algorithm with $\Oh(d^{3.5})$ running time, whereas the standard SDP-based approach requires $\Oh(d^{6.5})$ time.
We also provide a similar analysis for optimization problems over WSOS cones.

Although our approach circumvents the solution of the conventional semidefinite programming formulation of an SOS optimization problem, we show that an optimal solution to the equivalent semidefinite program can be recovered with little additional effort. This is necessary, for example, to construct explicit certificates proving that the optimal SOS polynomials are indeed SOS \cite{KaltofenLiYangZhi2008,PeyrlParrilo2008}.

The remainder of the paper is organized as follows: Section~\ref{sec:SY} describes the primal-dual non-symmetric conic optimization algorithm \cite{SkajaaYe2015} we use to optimize directly over SOS cones. Section~\ref{sec:barriers} discusses how the complexity of computing the Newton step depends on the bases $\vp$ and $\vq$ chosen for $\cV_{n,d}$ and $\cV_{n,2d}$ in Proposition~\ref{thm:Nesterov2000}. Section~\ref{sec:stability} concentrates on the case where $\vq$ is a Lagrange basis. We discuss how to choose the interpolation points and the basis $\vp$ to improve the numerical conditioning of the semidefinite representation of the SOS cone and compare the resulting representation against the alternatives. We also consider the question of how to choose an initial solution for the algorithm to improve its numerical stability. Section~\ref{sec:recovery} demonstrates how to recover explicit SOS decompositions for the optimal SOS polynomials that are computed using our approach. Section~\ref{sec:WSOS} generalizes the results of the previous sections to optimization problems over WSOS cones. Section~\ref{sec:experiments} presents the results of our computational experiments, which illustrate the advantages of the proposed approach over the standard SDP-based approach. Section~\ref{sec:conc} concludes the paper. We provide a brief review of barrier functions in Appendix \ref{sec:LHSCB} for completeness.
}

\section{Non-Symmetric Conic Optimization and the Skajaa--Ye Algorithm}\label{sec:SY}

Primal-dual interior-point methods are widely accepted as the most successful algorithms for conic optimization. The monographs \cite{NesterovNemirovski1994} and \cite{Renegar2001} provide a comprehensive overview of the rich theory behind these algorithms. However, the practical success of primal-dual interior-point methods has been largely limited to optimization problems over symmetric cones, which include linear programming, second-order cone programming, and semidefinite programming as special cases. The algorithms that have been developed for symmetric conic optimization are not directly applicable to SOS optimization because neither the SOS cone nor its dual are symmetric cones. 
Furthermore, most primal-dual interior-point methods for non-symmetric conic optimization, such as those proposed recently by Nesterov and others \cite{NesterovToddYe1999,Nesterov2012}, assume that a tractable logarithmically homogeneous self-concordant barrier (LHSCB) is known for both the primal cone $\cK$ and its dual $\cK^*$. This is not the case for SOS optimization. By Proposition~\ref{thm:Nesterov2000}, the cone $\Sigma_{n,2d}^*$ is a linear slice of a positive semidefinite cone; therefore, an LHSCB for this cone can be obtained by restricting the well-known logarithmic barrier $\vX\mapsto-\ln(\det(\vX))$ for the positive semidefinite cone to this particular slice. On the other hand, no similarly simple LHSCB is known for $\Sigma_{n,2d}$.

Our approach to SOS optimization is an adaptation of a recent primal-dual interior-point method by Skajaa and Ye for non-symmetric conic optimization \cite{SkajaaYe2015,PappYildiz2017corrigendum}. The key feature of this algorithm that makes it attractive for our purposes is that it requires only a tractable LHSCB for the primal cone $\cK$, but assumes nothing about the dual cone $\cK^*$. Hence, letting $\cK=\Sigma_{n,2d}^*$, this algorithm can be used to optimize over SOS polynomials directly. There are other methods with the same feature, e.g. \cite[Sec.~4.5]{NesterovNemirovski1994}; we have chosen to base our approach on the algorithm of Skajaa and Ye because it uses the homogeneous self-dual embedding. Algorithms that are based on the homogeneous self-dual embedding allow infeasible initial solutions, eliminating the need for a phase-I method, and have been used successfully in practice for optimization over symmetric cones.

\deletethis{Our approach to sum-of-squares optimization is an adaptation of a recent algorithm by Skajaa and Ye \cite{SkajaaYe2015,PappYildiz2017corrigendum}, a primal-dual interior-point method for general, non-symmetric, conic optimization that we shall briefly outline below. Most interior-point methods for conic optimization make use of a special class of functions called \emph{logarithmically homogeneous self-concordant barriers}, or LHSCBs for short. (See Appendix \ref{sec:LHSCB} for a brief review of LHSCBs necessary to keep this paper self-contained, and \cite{NesterovNemirovski1994} and \cite{Renegar2001} for a comprehensive overview of the mathematical theory of interior-point methods in convex conic optimization.) Most primal-dual interior-point methods for (\ref{eq:CP-P}-\ref{eq:CP-D}), such as those proposed recently by Nesterov and others \cite{NesterovToddYe1999,Nesterov2012}, assume that an efficiently computable LHSCB is known for both $\cK$ and $\cK^*$. This is not the case for sum-of-squares optimization. By Proposition~\ref{thm:Nesterov2000}, the cone $\Sigma_{n,2d}^*$ is a linear slice of a positive semidefinite cone; therefore, an LHSCB for this cone can be obtained by restricting the well-known logarithmic barrier $\vX\mapsto-\ln(\det(\vX))$ for the positive semidefinite cone to this particular slice. On the other hand, no similarly simple LHSCB is known for $\Sigma_{n,2d}$. The key feature of Skajaa and Ye's method that makes it attractive for our purposes is that it requires only a tractable LHSCB for the primal cone $\cK$, but assumes nothing about the dual cone $\cK^*$. Hence, letting $\cK=\Sigma_{n,2d}^*$, this algorithm can be used to optimize over SOS polynomials directly.}

\deletethis{
The study of theoretically and practically efficient algorithms for conic optimization (and general convex optimization) has largely focused on interior-point methods \cite{ForsgrenGillWright2002,NemirovskiTodd2008}, starting with Karmarkar's algorithm \cite{Karmarkar1984} for linear programming. This was generalized by Alizadeh \cite{Alizadeh1995} to semidefinite programming, by Nemirovskii and Scheinberg \cite{NemirovskiiScheinberg1996} to second-order cone programming, and by Nesterov and Todd \cite{NesterovTodd1997,NesterovTodd1998} to optimization over self-scaled cones, which were later discovered to be the same as symmetric cones \cite{Guler1996,Faybusovich1997,FarautKoranyi1994,Koecher1999}. This mathematically rich theory of interior-point methods was first developed in detail in \cite{NesterovNemirovski1994}, to be revisited and streamlined later in \cite{Renegar2001}. Note that the primal-dual machinery established in these works is only applicable to symmetric cones such as the positive semidefinite cone, which are highly special cases of self-dual cones. In particular, neither the cone of nonnegative polynomials nor the SOS cone nor their duals are symmetric (or even self-dual). Hence, a generalization of this theory is necessary to optimize directly over SOS cones, despite the fact that both the SOS cone and its dual are semidefinite representable.


Most interior-point methods for conic optimization make use of a special class of functions called \emph{logarithmically homogeneous self-concordant barriers}, or LHSCBs for short. (See Appendix~A for a brief review of LHSCBs and related notions.) In general, primal-dual interior-point methods for (\ref{eq:CP-P}-\ref{eq:CP-D}) assume that $\cK$ and $\cK^*$ are proper cones and require LHSCBs for both. The first such algorithm for general (non-symmetric) conic optimization is due to Nesterov, Todd, and Ye \cite{NesterovToddYe1999}. When applied to symmetric cones, this algorithm achieves the same asymptotic complexity as the best known methods for symmetric conic optimization, but it can also be applied to non-symmetric cones $\cK$ as long as both $\cK$ and $\cK^*$ admit LHSCBs whose value, gradient, and Hessian can all be computed efficiently. Nesterov's recent algorithm \cite{Nesterov2012} relaxes some of these requirements by requiring in each iteration the computation of the gradient and Hessian of an LHSCB for $\cK$ and the function values of an LHSCB for $\cK^*$; the derivatives of the dual LHSCB are not needed. Although these two algorithms are important theoretical breakthroughs with manifold applications outlined in \cite{Nesterov2012}, they cannot be used to optimize directly over SOS cones since no efficiently computable LHSCBs are known for SOS cones, even when the polynomials under consideration are univariate.

Recently, Skajaa and Ye \cite{SkajaaYe2015,PappYildiz2017corrigendum} proposed another primal-dual interior-point method for non-symmetric conic optimization. The most important feature of this algorithm that sets it apart from its predecessors is that it requires only a tractable LHSCB for the primal cone $\cK$, but assumes nothing about the dual cone $\cK^*$.
Hence, letting $\cK=\Sigma_{n,2d}^*$, this algorithm can be used to optimize over SOS polynomials directly.
Since according to Proposition~\ref{thm:Nesterov2000}, the cone $\Sigma_{n,2d}^*$ is not only semidefinite representable but also a linear slice of a positive semidefinite cone, an LHSCB for this cone can be obtained as a restriction of the well-known logarithmic barrier $\vX\mapsto-\ln(\det(\vX))$ for the positive semidefinite cone.
}


{\centering
	\begin{minipage}[tbh]{\linewidth}
		\begin{algorithm}[H]
			\caption{Predictor-Corrector Algorithm for Non-Symmetric Conic Optimization}
			\label{alg:SY}
			\begin{algorithmic}
				\STATE \hspace{-1em}\textbf{Parameters:} Real numbers $0<\eta<1$ and $\af_p,\af_c>0$ and an integer $r_c>0$ chosen according to \cite{PappYildiz2017corrigendum}.
				\STATE \hspace{-1em}\textbf{Input:} An LHSCB $F$ for $\cK$ and an initial solution $\vz^0 = (\bx^0,\vy^0,\bs^0) \in \cN(\eta)$.
				\LOOP
				\STATE \textbf{Termination?}
				\STATE If termination criteria are satisfied, stop and return $\vz=(\bx,\vy,\bs)$.
				\STATE \textbf{Prediction}
				\STATE Compute the Hessian $\bH(\bx)$ of $\bF$ at $\bx$, and solve for $\dz=(\dx,\dy,\ds)$ the  system
				 \begin{equation*}
					\begin{aligned}
					 \vG(\dy,\dx)-(\vzero,\ds)&=-(\vG(\vy,\bx)-(\vzero,\bs)),\\
					 \ds+\mu(\vz)\bH(\bx)\dx&=-\bs.
					\end{aligned}
				\end{equation*}\vspace{-1em}
				\STATE Set $\vz\leftarrow \vz+\af_p\dz$.
				\STATE \textbf{Correction}
				\FORALL{$i=1,\ldots,r_c$}
				\STATE Compute the gradient $\bg(\bx)$ and the Hessian $\bH(\bx)$ of $\bF$ at $\bx$, and solve for  $\dz=(\dx,\dy,\ds)$ the system
				 \begin{equation*}
					\begin{aligned}
					 \vG(\dy,\dx)-(\vzero,\ds)&=\vzero,\\
					 \ds+\mu(\vz)\bH(\bx)\dx&=-\psi(\vz).
					\end{aligned}
				\end{equation*}\vspace{-1em}
				\STATE Set $\vz\leftarrow \vz+\af_c\dz$.
				\ENDFOR
				\ENDLOOP
			\end{algorithmic}
		\end{algorithm}
	\end{minipage}
	\vspace{18pt}
}

In the rest of this section, we describe Skajaa and Ye's interior-point method; a pseudocode of the algorithm can be found in Algorithm~\ref{alg:SY}. (We omit a description of the termination criteria for the sake of brevity; see Section~5.4 of \cite{SkajaaYe2015} for details.) The algorithm requires an LHSCB $F:\cK^\circ\to\R$ for the proper cone $\cK$ whose gradient $g$ and Hessian $H$ can be computed efficiently at every point in $\cK^\circ$ and returns an $\varepsilon$-feasible solution to the so-called \emph{homogeneous self-dual embedding} of the problems (\ref{eq:CP-P}-\ref{eq:CP-D}).
Introducing two new scalar variables $\tau,\kappa\geq 0$, this homogeneous self-dual embedding is written as
\begin{gather}\label{eq:embed}
\begin{aligned}
                &           &\quad \vA  &\vx    & -\vb  & \tau  &&          &&          &&=\vzero\\
-\vA^\T         &\vy        &           &       & +\vc  & \tau  &&-\vs      &&          &&=\vzero\\
\quad \vb^\T    &\vy        &-\vc^\T    &\vx    &       &       &&          &&-\kappa   &&=0
\end{aligned}\\
\begin{gathered}
\vy\in\real^k,\quad (\vx,\tau)\in\cK\times\real_+,\quad (\vs,\kappa)\in\cK^*\times\real_+.\notag
\end{gathered}
\end{gather}
Let $\bx\defeq(\vx,\tau)$, $\bs\defeq(\vs,\kappa)$, $\bK\defeq\cK\times\real_+$, and $\bK^*\defeq\cK^*\times\real_+$. If the barrier parameter of $F$ is $\nu$, then the function $\bF(\bx)\defeq F(\vx)-\ln\tau$ is an LHSCB for the cone $\bK$ with barrier parameter $\bnu\defeq\nu+1$. Furthermore, its gradient and Hessian are $\bg(\bx) \defeq (g(\vx),-1/\tau)$ and $\bH(\bx)\defeq \left(\begin{smallmatrix}H(\vx) & \vzero\\ \vzero &1/\tau^2\end{smallmatrix}\right)$, respectively.

We state the precise result regarding the iteration complexity of Algorithm~\ref{alg:SY} next. To make this statement simpler, we let
\[
\vz\defeq(\bx,\vy,\bs),\qquad
\cF\defeq \bK\times\real^k\times\bK^*,
\quad\text{and}\quad
\vG\defeq\left(
\begin{smallmatrix}
\vzero      & \vA       & -\vb\\
-\vA^\T     & \vzero    &  \vc\\
\vb^\T      &-\vc^\T    &  0
\end{smallmatrix}\right)
.\]
Using this notation, the homogeneous self-dual embedding \eqref{eq:embed} can be expressed compactly as
\begin{equation*}
\vG(\vy,\bx)-(\vzero,\bs)=(\vzero,\vzero)\quad\text{and}\quad \vz=(\bx,\vy,\bs)\in\cF.
\end{equation*}
We let $\mu(\vz)\defeq\bx^\T\bs/\bnu$ denote the \emph{complementarity gap} of $\vz$. We also let $\psi(\vz)\defeq\bs+\mu(\vz)\bg(\bx)$.
For $0\leq\theta<1$, we define the \emph{$\theta$-neighborhood of the central path} for \eqref{eq:embed} as
\begin{equation*}
\cN(\theta)\defeq\left\{\vz=(\bx,\vy,\bs)\in\cF^\circ\,\vert\;\left\|\bH(\bx)^{-1/2}\psi(\vz)\right\|\leq\theta\mu(\vz)\right\}.
\end{equation*}

We refer the reader to Section 4.2 of \cite{SkajaaYe2015} for a formal description of the central path.

Algorithm~\ref{alg:SY} alternates between a predictor phase and a corrector phase until an $\varepsilon$-feasible solution to \eqref{eq:embed} is found. Each corrector phase consists of $r_c>0$ consecutive corrector steps. At each predictor and corrector step, the update direction is computed solving a linear system, and the current solution is updated along this direction using the step length $\af_p>0$ in prediction and the step length $\af_c>0$ in correction. With appropriately chosen parameters $\af_p,\af_c,r_c$, and $0<\eta<1$, the algorithm maintains the invariants that the predictor step updates a solution $\vz\in\cN(\eta)$ to a solution $\vz^+\in\cN(\beta)$ for some constant $0<\beta<1$ and the sequence of $r_c$ corrector steps update a solution $\vz\in\cN(\beta)$ to a solution $\vz^+\in\cN(\eta)$.
The following result shows that the parameters for Algorithm~\ref{alg:SY} can be chosen to ensure that the infeasibility and complementarity gap of \eqref{eq:embed} are reduced by a factor of $\varepsilon>0$ in $\Oh(\nu^{0.5}\log\varepsilon^{-1})$ iterations.


\begin{proposition}[\cite{SkajaaYe2015,PappYildiz2017corrigendum}]\label{thm:alg-complex}
For any $\varepsilon>0$, the parameters $\eta$, $\af_p$, $\af_c$, and $r_c$ can be chosen such that, given any initial solution $\vz^0=(\bx^0,\vy^0,\bs^0)\in\cN(\eta)$, Algorithm~\ref{alg:SY} terminates with a solution $\vz^*=(\bx^*,\vy^*,\bs^*) \in\cN(\eta)$ that satisfies
\begin{equation}\label{eq:alg-complex}
\mu(\vz^*)\leq\varepsilon\mu(\vz^0)\quad\text{and}\quad\|\vG(\vy^*,\bx^*)-(\vzero,\bs^*)\|\leq\varepsilon\|\vG(\vy^0,\bx^0)-(\vzero,\bs^0)\|
\end{equation}
in $\Oh(\nu^{0.5}\log\varepsilon^{-1})$ iterations.
\end{proposition}

Note that the step size $\af_p$ is a fixed parameter in the predictor step of Algorithm~\ref{alg:SY} as stated. However, the analysis of the algorithm and the iteration complexity result stated in Proposition~\ref{thm:alg-complex} are also applicable to the variant that instead uses line search to compute the (approximately) largest $\af_p$ for which $z+\af_p\dz\in\cN(\beta)$.

It can be shown that if the problems (\ref{eq:CP-P}-\ref{eq:CP-D}) are both feasible and have a zero duality gap, then Algorithm~\ref{alg:SY} returns a final solution $\vz^*$ with $\tau^*>0$, and the vectors $\vx^*/\tau^*$ and $(\vy^*,\vs^*)/\tau^*$ are approximately optimal solutions to (\ref{eq:CP-P}-\ref{eq:CP-D}), respectively. On the other hand, if one or both of the problems (\ref{eq:CP-P}-\ref{eq:CP-D}) are infeasible, then the algorithm returns a final solution $\vz^*$ with $\kappa^*>0$, and certificates of infeasibility can be obtained. For additional details, the reader is referred to \cite[Lem.~1]{SkajaaYe2015} and the discussion that follows.

\section{Tractable Barrier Functions for the Dual SOS Cone}\label{sec:barriers}

The cone $\Sigma_{n,2d}^*$ is a linear slice of the positive semidefinite cone $\bS^L_+$. Therefore, a restriction of the logarithmic barrier function $\vX\mapsto-\ln(\det(\vX))$ used in semidefinite programming is an LHSCB for $\Sigma_{n,2d}^*$. Furthermore, the barrier parameter of this restriction cannot exceed the barrier parameter $L={n+d\choose n}$ of the original logarithmic barrier function (see, for instance, \cite[Thm.~2.3.2]{Renegar2001}). In the case $n=1$, Nesterov \cite{Nesterov2000} showed that $\Sigma_{n,2d}^*$ does not admit any LHSCBs with a barrier parameter smaller than $L=d+1$, and his argument extends to general $n$ in a straightforward manner. Thus, we arrive at the following result.

\begin{proposition}[\cite{Nesterov2000}]\label{thm:Sigma*-barrier}
Using the notation of Proposition~\ref{thm:Nesterov2000}, for every pair of bases $\vp$ of $\cV_{n,d}$ and $\vq$ of $\cV_{n,2d}$, and the corresponding operator $\Lambda$, the function $\vx\mapsto-\ln(\det(\Lambda(\vx)))$ is an LHSCB for the cone $\Sigma_{n,2d}^*$ with barrier parameter $L={n+d\choose n}$.
\end{proposition}

\revise{
\begin{corollary}\label{thm:alg-complex-corollary}
Using the notation of Propositions \ref{thm:alg-complex} and \ref{thm:Sigma*-barrier}, Algorithm~\ref{alg:SY} using the barrier function $\vx\mapsto-\ln(\det(\Lambda(\vx)))$ for the cone $\cK=\Sigma_{n,2d}^*$ terminates with a solution \deletethis{$\vz^*\in\cN(\eta)$} satisfying the conditions \eqref{eq:alg-complex} \deletethis{in Proposition \ref{thm:alg-complex}} in $\Oh(L^{0.5}\log\varepsilon^{-1})$ iterations.
\end{corollary}
}
\revise{Note that this iteration complexity is the same as that of standard primal-dual interior-point methods applied to semidefinite programming problems with a single $L\times L$ matrix variable (see Section \ref{sec:complexity}).}

Note that the barrier function $\vx\mapsto-\ln(\det(\Lambda(\vx)))$ depends explicitly not only on the basis $\vq$ used to represent the polynomials in $\Sigma_{n,2d}$, but also on the basis $\vp$ in which the ``polynomials to be squared'' are represented. These basis choices greatly affect whether the barrier function and its derivatives can be evaluated in an efficient and numerically stable fashion. Depending on the subspace spanned by $\Lambda$, computing the gradient and Hessian of the barrier function may become the bottleneck of optimization over $\Sigma_{n,2d}$, or these computations may become ill-conditioned enough to make optimization over $\Sigma_{n,2d}$ impractical.


Let $F:(\Sigma_{n,2d}^*)^\circ\to\R$ be the function $F(\vx) \defeq -\ln(\det(\Lambda(\vx)))$. Recall that $\Lambda:\R^U\to\bS^L$ is a linear operator; therefore, there exist matrices $\vE_1,\ldots,\vE_U\in\bS^L$ such that $\Lambda(\vx) = \sum_{u=1}^U \vE_u x_u$. Using simple calculus, we get
\begin{equation}\label{eq:logdetgrad}
\frac{\partial F}{\partial x_u}(\vx) = -\Lambda(\vx)^{-1} \bullet \frac{\partial \Lambda(\vx)}{\partial x_u} = -\Lambda(\vx)^{-1} \bullet \vE_u, \qquad u=1,\ldots,U.
\end{equation}
Equivalently, $\nabla F(\vx)=-\Lambda^*(\Lambda(\vx)^{-1})$. To obtain the Hessian of $F$, we may start with the derivative of the inverse in differential form: $\dop\vX^{-1} = -\vX^{-1}(\dop\vX)\vX^{-1}$. This yields
\begin{equation}\label{eq:logdetHess}
\frac{\partial^2 F}{\partial
	x_u \partial x_v}(\vx) = \left(
\Lambda(\vx)^{-1}\vE_v\Lambda(\vx)^{-1}\right)\bullet\vE_u, \qquad u,v = 1, \ldots, U.
\end{equation}
That is to say, the Hessian $\nabla^2 F(\vx)$ is the linear operator that satisfies $\nabla^2 F(\vx)\vw=\Lambda^*\big(\Lambda(\vx)^{-1}\Lambda(\vw)\Lambda(\vx)^{-1}\big)$ for all $\vw\in\R^U$. The formulas (\ref{eq:logdetgrad}-\ref{eq:logdetHess}) indicate that the computation of the barrier gradient and Hessian can be inefficient and ill-conditioned.
In the remainder of this section, we consider natural basis choices for $\cV_{n,d}$ and $\cV_{n,2d}$ and compare the efficiency and numerical stability of computing the corresponding barrier function derivatives.

\subsection{Monomial basis}\label{sec:monomial-basis}

It is well-known (and easily derived from Proposition~\ref{thm:Nesterov2000}) that if $n=1$ and the bases $\vp$ and $\vq$ consist of monomials up to degree $d$ and $2d$ respectively, then $\Lambda$ is the mapping from $\vx$ to its \emph{Hankel matrix} $(x_{i+j})_{i,j=0,\ldots,d}$. The inverse of a positive definite Hankel matrix can be computed in $\Oh(d^2)$ time using specialized algorithms such as those described in \cite{HeinigJankowski1990} and \cite[Ch.~5]{Pan2001}.
Additionally, in this case $\vE_u = (\delta_{i+j,u})_{i,j=0,\ldots,d}$ for $u=0,\ldots,2d$, where $\delta_{i+j,u}$ denotes the Kronecker delta. The special structure of these matrices allows for numerous simplifications in the formulas (\ref{eq:logdetgrad}-\ref{eq:logdetHess}). With $\Lambda(\vx)^{-1}$ already computed, the gradient $\nabla F(\vx)$ can be calculated with $O(d^2)$ arithmetic operations using \eqref{eq:logdetgrad}. The calculation of the Hessian can also be accelerated. The argument below follows \cite[Thm.~5.27]{Papp2011}; similar solutions had also been proposed in \cite{AlkireVandenberghe2002,GeninHachezNesterovVanDooren2003}. Let $z_{ij}$ denote the $(i,j)$-th entry of $\Lambda(\vx)^{-1}$. Then
\[ \Lambda(\vx)^{-1}\vE_v\Lambda(\vx)^{-1} = \bigg(\sum_{k+l=v}z_{ak}z_{lb} \bigg)_{a,b=0,\ldots 2d}, \]
which in turn yields
\[\frac{\partial^2 F}{\partial x_u \partial x_v}(\vx)
= \left( \Lambda(\vx)^{-1}\vE_v\Lambda(\vx)^{-1}\right) \bullet \vE_u
= \sum_{\substack{a+b=u\\k+l=v}}z_{ak}z_{bl}.\]
The last summation shows that the Hessian $\nabla^2 F(\vx)$ is the convolution of $\Lambda(\vx)^{-1}$ with itself; equivalently, the Hessian is the coefficient matrix of the square of the bivariate polynomial whose coefficient matrix is $\Lambda(\vx)^{-1}$. If $\Lambda(\vx)^{-1}$ is already computed, this convolution can be computed with a single bivariate polynomial multiplication. This multiplication can be carried out in $\Oh(d^2\log d)$ arithmetic operations using two-dimensional fast Fourier transform. Hence, we have shown the following:

\begin{theorem}\label{thm:monomial-basis}
Using the notation of Proposition~\ref{thm:Nesterov2000}, if $n=1$ and the bases $\vp$ and $\vq$ consist of the monomials up to degree $d$ and $2d$ respectively, then the gradient and Hessian of the barrier $\vx\mapsto-\ln(\det(\Lambda(\vx)))$ can be computed in $O(d^2\log d)$ time.
\end{theorem}

However, the monomial basis representation is not suitable for problems involving high-degree polynomials because the resulting semidefinite representation is inherently ill-conditioned. In particular, the condition number of positive definite Hankel matrices increases exponentially with the dimension, and every positive definite Hankel matrix of order 40 or higher has a condition number greater than the reciprocal of machine epsilon in double precision \cite{Beckermann2000}. Therefore, any interior-point method that optimizes over $\Sigma_{1,2d}^*$ in the monomial basis representation and requires the solution of linear systems where $\Lambda(\vx)$ is the constraint matrix is unstable and impractical for even moderate degrees.
In addition, the scheme described above for computing the barrier gradient and Hessian exploits the fact that $\Lambda(\vx)$ is a Hankel matrix in the semidefinite representation of the univariate (unweighted) SOS polynomials, and it is not immediate to generalize this approach to multivariate or WSOS polynomials while also maintaining its efficiency.

\subsection{Chebyshev basis}\label{sec:chebyshev-basis}
The proof of Theorem~\ref{thm:monomial-basis} carries over to other bases $\vp$ and $\vq$ as long as the coefficients (in the basis $\vq$) of the square of a polynomial (given by its coefficients in the basis $\vp$) can be computed in $O(d^2\log d)$ time and the computations involving Hankel matrices can be replaced with analogous computations involving another family of structured matrices for which matrix inversion can be carried out in $O(d^2\log d)$ time. This is, for example, true when $n=1$ and the bases $\vp$ and $\vq$ consist of the Chebyshev polynomials (of the first kind) up to degree $d$ and $2d$ respectively. These are the polynomials defined according to the recursion
\[ T_0(t) = 1,\quad T_1(t) = t,\quad\text{and}\quad T_i(t) = 2tT_{i-1}(t) - T_{i-2}(t) \quad \forall\;i \geq 2.\]
Using the well-known identity $T_i T_j = \frac{1}{2}(T_{i+j} + T_{|i-j|})$ (see, e.g., \cite[Sec.~2.4]{MasonHandscomb2003}), we find that the operator $\Lambda$ corresponding to this choice of bases is
\[ \Lambda(\vx)_{ij} = \frac{x_{i+j} + x_{|i-j|}}{2} \quad i,j=0, \ldots, d. \]
Therefore, the matrix $\Lambda(\vx)$ is now a \emph{Toeplitz-plus-Hankel matrix}. The inverse of a positive definite Toeplitz-plus-Hankel matrix can be computed in $\Oh(d^2)$ time \cite[Ch.~5]{Pan2001}.
Furthermore, extensions of the fast Fourier transform to Chebyshev polynomials are also known \cite[Ch.~4]{MasonHandscomb2003}. Consequently, the algorithm outlined in the discussion preceding Theorem~\ref{thm:monomial-basis} extends to the case where both degree-$d$ and degree-$2d$ polynomials are represented in the Chebyshev basis, and its running time remains the same.

\begin{theorem}\label{thm:chebyshev-basis}
Using the notation of Proposition~\ref{thm:Nesterov2000}, if $n=1$ and the bases $\vp$ and $\vq$ consist of the Chebyshev polynomials of the first kind up to degree $d$ and $2d$ respectively, then the gradient and Hessian of the barrier $\vx\mapsto-\ln(\det(\Lambda(\vx)))$ can be computed in $\Oh(d^2\log d)$ time.
\end{theorem}

While the Chebyshev basis representation does address the numerical problems associated with the monomial basis, the efficient scheme described above for computing the barrier gradient and Hessian requires specialized techniques that take advantage of the Toeplitz-plus-Hankel structure of the matrix $\Lambda(\vx)$ which is encountered in the semidefinite representation of univariate (unweighted) SOS polynomials. As with the monomial basis, the generalization of this approach to multivariate or WSOS polynomials is not straightforward and will not be pursued in this paper.

\subsection{Interpolant basis}\label{sec:interpolant-basis}
Another approach to address the numerical difficulties that arise when using monomial bases is to use interpolating polynomials. Recall that a set of points in $\real^n$ is called \emph{unisolvent} for a linear space $\cV$ of $n$-variate polynomials if every polynomial in $\cV$ is uniquely determined by its function values at these points. If $n=1$, every set of $2d+1$ distinct points is unisolvent for $\cV_{n,2d}$, but this is no longer the case for $n\geq 2$.

Representing degree-$2d$ polynomials with their values at prescribed interpolation points, we reach the interpolant basis representation of $\Sigma_{n,2d}^*$. To make this concrete, let $\cT\defeq\{\vt_1,\ldots,\vt_U\}\subset\R^n$ be a unisolvent set for $\cV_{n,2d}$.
For $u=1,\ldots,U$, let $q_u\in\cV_{n,2d}$ be the Lagrange polynomial that satisfies $q_u(\vt_u)=1$ and $q_u(\vt_v)=0$ for every $v\neq u$.
Then $\vq=(q_1,\ldots,q_U)$ is a Lagrange basis for $\cV_{n,2d}$. The coefficients of any polynomial $f\in\cV_{n,2d}$ in this basis are precisely its function values at $\vt_1,\ldots,\vt_U$: $f=\sum_{u=1}^U f(\vt_u)q_u$. In particular, given any basis $\vp$ of $\cV_{n,d}$, we have $p_ip_j=\sum_{u=1}^U p_i(\vt_u)p_j(\vt_u)q_u$ for all $i,j=1,\ldots,L$. In matrix form, these equations can be expressed as $\vp\vp^\T=\vP^\T\diag(\vq)\vP$ where $\vP \defeq (p_\ell(\vt_u))_{u=1,\ldots,U;\ell=1,\ldots,L}$. Thus, the operator $\vx\mapsto\vP^\T\diag(\vx)\vP$ satisfies the condition $\vp\vp^\T=\Lambda(\vq)$ in Proposition~\ref{thm:Nesterov2000}, and because $\Lambda$ is the unique linear operator with this property, we have $\Lambda(\vx)=\vP^\T\diag(\vx)\vP.$
Accordingly, the cone $\Sigma_{n,2d}^*$ admits a semidefinite characterization as the set of points $\vx\in\R^U$ that satisfy $\vP^\T\diag(\vx)\vP\succcurlyeq\vzero$.




The operator $\Lambda$ has the expression $\Lambda(\vx)=\sum_{u=1}^U \vE_ux_u$ in terms of the matrices $\vE_u = \vp(\vt_u)\vp(\vt_u)^\T$. Taking advantage of the rank-one structure of these matrices, the adjoint of $\Lambda$ can be expressed as $\Lambda^*(\vS)=(\vE_u\bullet\vS)_{u=1,\ldots,U}=\diag(\vP\vS\vP^\T)$. Furthermore, the formulas (\ref{eq:logdetgrad}-\ref{eq:logdetHess}) for the gradient and Hessian of the barrier $F(\vx)=-\ln(\det(\Lambda(\vx)))$ can be simplified to
\begin{gather}
\nabla F(\vx) = -\diag\left(\vP\left(\vP^\T\diag(\vx)\vP\right)^{-1}\vP^\T\right)\quad\text{and}\label{eq:grad}\\
\nabla^2 F(\vx) = \left(\vP\left(\vP^\T\diag(\vx)\vP\right)^{-1}\vP^\T\right)^{\circ2},\label{eq:Hess}
\end{gather}
where $\vQ^{\circ2}$ denotes the elementwise (Hadamard) square of the matrix $\vQ$. These formulas allow the efficient computation of $\nabla F(\vx)$ and $\nabla^2 F(\vx)$ without structured matrix inversion or multivariate fast Fourier transform, in any dimension, as the next theorem demonstrates.

\begin{theorem}\label{thm:interpolant-basis}
Using the notation of Proposition~\ref{thm:Nesterov2000}, for every $n$ and $d$ and for every basis $\vp$ of $\cV_{n,d}$, if $\vq$ is a Lagrange basis for $\cV_{n,2d}$, then the gradient and Hessian of the barrier $\vx\mapsto-\ln(\det(\Lambda(\vx)))$ can be computed in $\Oh(LU^2)$ time, using $\Oh(LU)$ working memory in addition to the $\Oh(U^2)$ space required to store the Hessian.

In particular, when $n=1$, these computations take $\Oh(d^3)$ time using $\Oh(d^2)$ space.
\end{theorem}
\begin{proof}
We may assume that the matrix $\vP$ has been computed (offline) and stored in advance, using $\Oh(LU)$ space. Then the $L\times L$ matrix $\Lambda(\vx)=\vP^\T\diag(\vx)\vP$ can be computed in $\Oh(L^2U)$ arithmetic operations. The Cholesky factorization $\Lambda(\vx)=\vL\vL^\T$ can be performed in $\Oh(L^3)$ arithmetic operations \cite[Lec.~23]{TrefethenBau1997}. 
Then the matrix $\vV\defeq\vL^{-1}\vP^\T$ can be computed solving $U$ triangular systems with a total of $\Oh(L^2U)$ arithmetic operations. The matrix $\vQ\defeq\vP\Lambda(\vx)^{-1}\vP^\T$ can now be computed with an additional $\Oh(LU^2)$ arithmetic operations using $\vQ=\vV^\T\vV$. From (\ref{eq:grad}-\ref{eq:Hess}), the gradient and Hessian of the barrier are the negative of the diagonal of $\vQ$ and the elementwise square of $\vQ$ respectively; these can be computed from $\vQ$ in $\Oh(U^2)$ arithmetic operations.
\end{proof}

The observation that the gradient and Hessian of the barrier function $F$ can be computed efficiently in the interpolant basis representation was made earlier in \cite{LofbergParrilo2004}. In the context of Algorithm~\ref{alg:SY}, Theorem~\ref{thm:interpolant-basis} shows that using this representation, the algorithmic bottleneck at each iteration is the computation of the predictor and corrector directions, which require $\Oh(U^3)$ arithmetic operations, and not the computation of the barrier Hessian, with requires only $\Oh(LU^2)$ arithmetic operations using the procedure outlined in the proof. \revise{Therefore, each iteration of Algorithm~\ref{alg:SY} can be performed using $\Oh(U^3)$ arithmetic operations. In contrast, each iteration of a standard primal-dual interior-point method applied to the usual semidefinite programming formulation of an optimization problem over   $\Sigma_{n,2d}^*$ requires $\Oh(L^6)$ arithmetic operations (see Section \ref{sec:complexity}).}

In the remainder, we mainly focus on the interpolant basis representation of the cone $\Sigma_{n,2d}$. Besides allowing the efficient computation of the gradient and Hessian of the barrier function $F$, this representation has two additional advantages.
First, the approach described above for evaluating the derivatives of $F$ can be generalized to the weighted case in a straightforward fashion (see Section~\ref{sec:WSOS}). Second, the interpolant basis representation of $\Sigma_{n,2d}$ is numerically well-conditioned for appropriate choices of the basis $\vp$ and the set of interpolation points $\cT$. We discuss the latter point further in the next section.

\section{Conditioning and Stability}\label{sec:stability}

Proposition~\ref{thm:Nesterov2000} allows for an infinite family of representations of $\Sigma_{n,2d}$ and $\Sigma_{n,2d}^*$, parameterized with the bases $\vp$ and $\vq$. Even after identifying $\vq$ as the Lagrange basis corresponding to a set $\cT=\{\vt_1,\ldots,\vt_U\}$ which is unisolvent for $\cV_{n,2d}$, there is still flexibility in selecting a final representation as both the basis $\vp$ and interpolation points $\cT$ can be chosen rather freely. The complexity of computing the barrier gradient and Hessian is independent of these choices, but the numerical conditioning of the representation naturally depends on them. In this section, we investigate how the selection of $\vp$ and $\cT$ affects the numerical stability of Algorithm~\ref{alg:SY} when solving SOS optimization problems represented using the interpolant basis.

The conditioning of optimization over $\Sigma_{n,2d}$ and $\Sigma_{n,2d}^*$ is closely related to the condition numbers of the linear operators $\Lambda$ and $\Lambda^*$ used in their semidefinite representation.
Recall that $\Lambda(\vx)=\sum_{u=1}^U \vE_ux_u$ for some matrices $\vE_1,\ldots,\vE_U\in\bS^L$. Then $\Lambda^*$ has the expression $\Lambda^*(\vS)=(\vE_u\bullet \vS)_{u=1,\ldots,U}$. Equivalently, $\Lambda^*(\vS) = \vLambda^\T\vec(\vS)$,
where $\vec(\vS)$ is the column vector obtained by ``stacking'' the columns of $\vS$, and $\vLambda$ is the $L^2 \times U$ matrix $\vLambda\defeq \big(\vec\big(\vE_u\big)\big)_{u=1,\ldots,U}$. Hence, the condition number of the operators $\Lambda$ and $\Lambda^*$ is the condition number of the matrices $\vLambda$ and $\vLambda^\T$.

Let $H(\vx)\defeq\nabla^2 F(\vx)$ denote the Hessian of the barrier function $F(\vx)=-\ln(\det(\Lambda(\vx)))$. Theorem~\ref{thm:lambdaHx} below shows that the condition number of $H(\vx)$ can be bounded from above using the condition numbers of $\vLambda$ and $\Lambda(\vx)$.
This analysis uses Lemma~\ref{thm:trAB}, which we present first. We let $\lambdamin(\vQ)$ and $\lambdamax(\vQ)$ denote the smallest and largest eigenvalues of a real symmetric matrix $\vQ$. We also let $\cond(\vQ)$ denote the condition number of a real matrix $\vQ$ relative to the $\ell_2$-norm. Recall that $\cond(\vQ)$ equals the ratio of the largest singular value of $\vQ$ to its smallest singular value.

\begin{lemma}\label{thm:trAB}
Let $\vA\in\bS^L_+$ and $\vB\in\bS^L$. Then
\[ \tr(\vA)\lambdamin(\vB) \leq \tr(\vA\vB) \leq \tr(\vA)\lambdamax(\vB). \]
\end{lemma}

\begin{theorem}\label{thm:lambdaHx}
Using the notation of Proposition~\ref{thm:Nesterov2000}, for every pair of bases $\vp$ of $\cV_{n,d}$ and $\vq$ of $\cV_{n,2d}$, and the corresponding operator $\Lambda$, one has
\begin{equation}\label{eq:lambdaHx}
\cond(H(\vx))\leq\cond(\vLambda)^2\cond(\Lambda(\vx))^2.
\end{equation}
\end{theorem}
\begin{proof}
From \eqref{eq:logdetHess}, note that $\vw^\T H(\vx)\vw=\tr(\Lambda(\vw)\Lambda(\vx)^{-1}\Lambda(\vw)\Lambda(\vx)^{-1})$. Applying Lemma~\ref{thm:trAB} repeatedly, we obtain
\begin{align*}
\vw^\T H(\vx)\vw&\geq\lambdamin(\Lambda(\vx)^{-1})\tr(\Lambda(\vw)\Lambda(\vx)^{-1}\Lambda(\vw))\\
&=\lambdamin(\Lambda(\vx)^{-1})\tr(\Lambda(\vw)^2\Lambda(\vx)^{-1})\\
&\geq\lambdamin(\Lambda(\vx)^{-1})^2\tr(\Lambda(\vw)^2)\\
&=\lambdamax(\Lambda(\vx))^{-2}\tr(\Lambda(\vw)^2).
\end{align*}
Now the observation that $\tr(\Lambda(\vw)^2)=\vw^\T\vLambda^\T\vLambda\vw\geq \lambdamin(\vLambda^\T\vLambda)\|\vw\|^2$ shows $\lambdamin(H(\vx)) \geq \tfrac{\lambdamin(\vLambda^\T\vLambda)}{\lambdamax(\Lambda(\vx))^2}$. Analogously, we get $\lambdamax(H(\vx)) \leq \tfrac{\lambdamax(\vLambda^\T\vLambda)}{\lambdamin(\Lambda(\vx))^2}$. Combining these inequalities and using the identity $\cond(\vLambda)^2=\tfrac{\lambdamax(\vLambda^\T\vLambda)}{\lambdamin(\vLambda^\T\vLambda)}$ yields our claim.
\end{proof}

Theorem~\ref{thm:lambdaHx} reveals the connection between the inherent conditioning of the chosen SOS representation and the conditioning of the barrier Hessian throughout the algorithm. It is expected that as the algorithm progresses and $\vx$ converges to the boundary of $\Sigma_{n,2d}^*$, the matrix $\Lambda(\vx)$ becomes increasingly singular and ill-conditioned. Theorem~\ref{thm:lambdaHx} shows that the condition number of the Hessian $H(\vx)$ increases proportionately to the square of the condition number of $\Lambda(\vx)$, and the proportionality constant is $\cond(\vLambda)^2$.

In the remainder of this section, we concentrate on the case where $\vq$ is a Lagrange basis. Recall that in this case $\Lambda(\vx)$ has the expression $\Lambda(\vx)=\vP^\T\diag(\vx)\vP$ where $\vP =(p_\ell(\vt_u))_{u=1,\ldots,U;\ell=1,\ldots,L}$. Furthermore, the corresponding $\vLambda$ matrix satisfies $\vLambda^\T\vLambda=\PPTsq$.

\subsection{Selection of interpolation points and bases}\label{sec:interpolation-points}

The question of how to choose good points for polynomial interpolation is a difficult but well-studied problem. A desirable property of interpolation points is that the constructed polynomial interpolants are not highly sensitive to the prescribed function values at the interpolation points. A natural measure of this sensitivity is the \emph{Lebesgue constant}, which is the condition number of the interpolation operator (mapping the function values at the interpolation points to the interpolating polynomial) with respect to the $\ell_\infty$-norm (see, e.g., \cite[Ch.~15]{Trefethen2013}). 
Explicit formulas for families of interpolation points that have asymptotically optimal Lebesgue constants for total-degree polynomial interpolation are known only for certain low-dimensional standard domains. Some prominent examples are \emph{Chebyshev points} on bounded intervals and \emph{Padua points} on rectangular domains \cite{CaliariDeMarchiVianello2005}. For univariate polynomial interpolation of degree $d$, Chebyshev points of the first kind are defined on $[-1,1]$ as
\begin{equation}\label{eq:ChebPts-1stkind}
\cC_{1,d} \defeq \left\{\cos((\ell+0.5)\pi/(d+1))\;|\;\ell=0,\ldots,d\right\},
\end{equation}
whereas Chebyshev points of the second kind are defined as
\begin{equation}\label{eq:ChebPts-2ndkind}
\cC_{2,d} \defeq \left\{\cos(\ell\pi/d)\;|\;\ell=0,\ldots,d\right\}.
\end{equation}
Let $\cC_{2,d}^E$ and $\cC_{2,d}^O$ denote the subsets of $\cC_{2,d}$ consisting of the Chebyshev points with even indices and odd indices respectively. For bivariate polynomial interpolation of (total) degree $d$, Padua points are defined on $[-1,1]^2$ as
\begin{equation}\label{eq:PaduaPts}
\cP_d \defeq \left(\cC_{2,d}^E\times\cC_{2,d+1}^O\right)\cup\left(\cC_{2,d}^O\times\cC_{2,d+1}^E\right).
\end{equation}
We note that both Chebyshev and Padua points can be adapted to arbitrary intervals and two-dimensional rectangular domains via an affine change of variables. The Lebesgue constant for degree-$d$ polynomial interpolation on $[-1,1]$ using Chebyshev points is of order $\Oh(\log d)$, and this growth is asymptotically optimal \cite[Ch.~15]{Trefethen2013}. Similarly, The Lebesgue constant for degree-$d$ polynomial interpolation on $[-1,1]^2$ using Padua points is $\Oh(\log^2 d)$, and this growth is asymptotically optimal \cite{BosCaliariDeMarchiVianelloXu2006}.

Another useful family of interpolation points is \emph{Fekete points}. Let $\vr=(r_1,\ldots,r_U)$ be a basis of $\cV_{n,2d}$. The Fekete points associated with a compact domain $\cD\subset\R^n$ are the points $\vt_1,\ldots,\vt_U\in\cD$ that maximize the absolute value of the determinant of the Vandermonde matrix $(r_\ell(\vt_u))_{u=1,\ldots,U;\ell=1,\ldots,U}$. Note that these points are independent of the basis choice $\vr$ because any basis change multiplies the determinant by a constant nonzero factor. While Fekete points are well-defined for any compact domain in any dimension, they are known analytically only for certain special domains such as the interval \cite{BosTaylorWingate2001} and are hard to compute numerically in general \cite{TaylorWingateVincent2000}.

For interpolation in higher dimensions and on more general domains,
one can compute \emph{approximate Fekete points}, following an approach due to Sommariva and Vianello \cite{SommarivaVianello2009}. The underlying idea of this approach is to first extract a large but finite number of candidate points from the domain and then to choose from these candidate points a subset that approximately maximizes the absolute value of the Vandermonde determinant. We now make this more concrete.
For $N\gg U$, let $\cT'=\{\vt_1,\ldots,\vt_N\}$ be a unisolvent set for $\cV_{n,2d}$, and let $\vr=(r_1,\ldots,r_U)$ be a basis of $\cV_{n,2d}$. The $N\times U$ Vandermonde matrix $\vV\defeq (r_u(\vt_i))_{i=1,\ldots,N;u=1,\ldots,U}$ has linearly independent columns because $\vr$ is linearly independent and $\cT'$ is unisolvent for $\cV_{n,2d}$. Ideally, we would like to find the $U\times U$ row submatrix of $\vV$ that has the largest absolute determinant. However, this problem in NP-hard in general \cite{Papadimitriou1984}; therefore, we resort to a greedy heuristic \cite{Khachiyan1995} to choose an appropriate subset of rows. The main step of this algorithm can be performed via QR factorization with column pivoting \cite{SommarivaVianello2009}.
The resulting Vandermonde submatrix is nonsingular, and the subset of points corresponding to the selected rows is automatically unisolvent for $\cV_{n,2d}$.

In this paper, Chebyshev points of the second kind and Padua points are used for interpolation on $[-1,1]$ and $[-1,1]^2$, respectively. We use the Matlab package Chebfun \cite{chebfun} to compute Chebyshev points and Padua2DM \cite{padua2dm} to compute Padua points. On the hypercube $[-1,1]^n$ for $n>2$, we generate approximate Fekete points for degree-$2d$ interpolation. In the definition of the Vandermonde matrix $\vV$ above, we let $r_1,\ldots,r_U\in\cV_{n,2d}$ be the multivariate Chebyshev polynomials of degree at most $2d$, and we let $\cT'$ be the product Chebyshev grid $\cC_{2,2d+1}\times\ldots\times\cC_{2,2d+n}$.

Once the interpolation points have been fixed, we choose a basis $\vp=(p_1,\ldots,p_L)$ of $\cV_{n,d}$ to ensure
that the resulting matrix $\vP$ has orthonormal columns. For this, we do not need to determine such a basis in closed form; orthonormalizing the columns of the matrix $(p_\ell'(\vt_u))_{u=1,\ldots,U;\ell=1,\ldots,L}$ associated with any basis $\vp'=(p_1',\ldots,p_L')$ produces a matrix $\vP$ of the desired type, and the polynomials $p_1,\ldots,p_L$ are then defined implicitly via their function values at the interpolation points.

\subsection{Initialization}

Depending on the chosen interpolant representation and the resulting matrix $\vP$, a good initial solution $\vz^0=(\bx^0,\vy^0,\bs^0)$ can help the matrix $\Lambda(\vx)=\vP^\T\diag(\vx)\vP$ and the Hessian $H(\vx)$ remain sufficiently well-conditioned for all but the last few iterations of Algorithm~\ref{alg:SY}. To inform our selection of the initial solution, we make use of Theorem~\ref{thm:lambdaHx}: for optimization over a single cone $\Sigma_{n,2d}$, assuming (without loss of generality) that the matrix $\vP$ has orthonormal columns, choosing $\vx^0=\delta\vone$ for some $\delta>0$ results in $\Lambda(\vx^0)=\delta\vI$. The latter matrix is positive definite, which ensures that $\vx^0\in(\Sigma_{n,2d}^*)^\circ$, and it has perfect conditioning, which means that $\vx^0$ minimizes the right-hand side of \eqref{eq:lambdaHx} for the given $\Lambda$.
Then the initial Hessian becomes $H(\vx^0)=\delta^{-2}\vLambda^\T\vLambda=\delta^{-2}\PPTsq$.
Given $\vx^0=\delta\vone$, we choose the remaining variables according to $\vs^0=-g(\vx^0)=-\delta^{-1}g(\vone)$, $\tau^0=\kappa^0=1$, and $\vy^0=\vzero$. Note that $\vs^0\in(\Sigma_{n,2d})^\circ$ because $\vx^0\in(\Sigma_{n,2d}^*)^\circ$ (see, e.g., \cite[Theorem~3.3.1]{Renegar2001}). Furthermore, because $F$ is an LHSCB with barrier parameter $\nu$, its gradient satisfies $\vx^\T g(\vx) = -\nu$ for every $\vx\in\cK^\circ$ (see, e.g., \cite[Theorem~2.3.9]{Renegar2001}), which yields that $\mu(\vz^0)=1$. Finally, straightforward arithmetic shows that $\psi(\vz^0)=0$ and $\vz^0\in\cN(0)$. Hence, $\vz^0$ is a fairly well-conditioned, strictly interior initial solution in the $0$-neighborhood of the central path.

In this framework, there is still flexibility with respect to the choice of $\delta>0$. In our implementation, we use $\delta=\sqrt{\delta_P\delta_D}$ where
\[
\delta_P=\max_{i=1,\ldots,k}\frac{1+|\ve_i^T\vb|}{1+|\ve_i^T\vA\vone|}\quad\text{and}\quad
\delta_D=\max_{j=1,\ldots,U}\frac{1+|\ve_j^Tg(\vone)|}{1+|\ve_j^T\vc|}.
\]

This choice produces an initial solution with moderate primal and dual infeasibility values.

\subsection{Comparison to the monomial basis}

The good conditioning of the interpolant basis representation of SOS constraints is in sharp contrast with their traditional representation in the monomial basis. Recall from Section~\ref{sec:monomial-basis} that for the monomial basis, $\Lambda(\vx)$ is the Hankel matrix of the vector $\vx$ in the univariate case, which is ill-conditioned even for polynomials of moderate degree, meaning that the Hessian cannot be computed accurately at any iteration of the algorithm.
Using the interpolant basis and an initial point as described above, the matrix $\Lambda(\vx)$ has a perfect condition number at the start of the algorithm, the initial Hessians are not much worse conditioned than $\PPTsq$, and with the appropriate choice of $\vP$, the computation of the predictor and corrector steps remains stable.

\subsection{Comparison to orthogonal bases}

In contrast with the monomial basis, the Chebyshev basis and other orthogonal bases are popular in numerical algorithms for high-degree polynomials. We shall show that for bases satisfying a discrete orthogonality condition, the interpolant basis representation of SOS polynomials (with appropriately chosen interpolation points) is not worse conditioned than the orthogonal basis representation.

Consider a basis $\hat\vq=(\hat{q}_1,\ldots,\hat{q}_U)$ of $\cV_{n,2d}$ and a set $\cT=\{\vt_1,\ldots,\vt_U\}$ which is unisolvent for $\cV_{n,2d}$. Suppose that $\hat\vq$ satisfies the discrete orthogonality conditions with respect to $\cT$:
\begin{equation}\label{eq:disc-ortho}
\sum_{u=1}^U \hat{q}_i(\vt_u)\hat{q}_j(\vt_u) = \begin{cases} 1 & \text{if } i = j,   \; 1\leq i,j\leq U \\
                                                              0 & \text{if } i \neq j,\; 1\leq i,j\leq U.\end{cases}
\end{equation}
Let $\vq$ be the Lagrange basis corresponding to $\cT$. Throughout this section, our primordial example is the case where $\hat\vq$ consists of the normalized Chebyshev polynomials of the first kind, defined as $\hat{T}_0(t) = \sqrt{\tfrac{1}{d+1}}$ and $\hat{T}_i(t) = \sqrt{\tfrac{2}{d+1}}T_i(t)$ for $i=1,\dots,d$, and $\cT$ is the set of Chebyshev points $\cC_{1,2d}$ defined in \eqref{eq:ChebPts-1stkind}. In this case, it is known that \eqref{eq:disc-ortho} holds (see, e.g., \cite[Sec.~4.6.1]{MasonHandscomb2003}).

As before, let $\vp$ be arbitrary; let $\Lambda$ denote the operator in Proposition~\ref{thm:Nesterov2000} corresponding to the bases $\vp$ and $\vq$, and let $\hat\Lambda$ denote the operator corresponding to $\vp$ and $\hat\vq$. Recall from the properties of the Lagrange basis that $\hat{q}_j = \sum_{u=1}^U\hat{q}_j(\vt_u)q_u$ for $j=1,\dots,U$. Using the Vandermonde matrix $\vV\defeq (\hat{q}_j(\vt_u))_{u=1,\ldots,U;j=1,\ldots,U}$, these equations can be written in matrix form as $\hat{\vq} = \vV^\T\vq$, and the definitions of $\Lambda$ and $\hat{\Lambda}$ yield $\vp\vp^\T = \Lambda(\vq) = \hat\Lambda(\hat\vq) = \hat\Lambda(\vV^\T\vq)$. Therefore, $\cond(\Lambda) \leq \cond(\hat{\Lambda})\cond(\vV)$. If, in addition, $\hat\vq$ satisfies the discrete orthogonality conditions \eqref{eq:disc-ortho}, then $\vV^\T\vV = \vI$ and the condition number of $\vV$ is one; hence, the inequality above simplifies to $\cond(\Lambda) \leq \cond(\hat{\Lambda})$.

\deletethis{
Consider a basis $\vr=(r_1,\ldots,r_U)$ of $\cV_{n,2d}$ that is different from the Lagrange basis $\vq$. For every pair $(i,j)\in\{1,\ldots,L\}^2$, there exists a unique $\vc_{i,j}\in\R^U$ that satisfies $p_ip_j=\vc_{i,j}^\T\vr$; let $\vC$ be the $L^2\times U$ matrix whose row indexed with $(i,j)$ is $\vc_{i,j}^\T$. By its definition, the matrix $\vC$ satisfies $\vC\vr=\vec(\vp\vp^\T)$. Therefore, according to Proposition~\ref{thm:Nesterov2000}, it defines a representation of the SOS cone as $\Sigma_{n,2d}=\{\vs\in\R^U\,|\,\vs=\vC^\T\vS,\,\vS\succcurlyeq\vzero\}$. Furthermore, based on the unisolvence of the chosen interpolation points, note that $\vr=\vR^\T\vq$ where $\vR\defeq(r_j(\vt_u))_{u=1,\ldots,U;j=1,\ldots,U}$. Together with $\vec(\vp\vp^\T)=\vLambda\vq$, this implies $(\vLambda-\vC\vR^\T)\vq=\vzero$, and because the polynomials $q_1,\ldots,q_U$ are linearly independent, we get $\vLambda^\T=\vR\vC^\T$. It follows that $\cond(\vLambda)\leq\cond(\vR)\cond(\vC)$.

Suppose now that the polynomials $r_1,\ldots,r_U$ are mutually orthogonal with respect to the interpolation points $\vt_1,\ldots,\vt_U$. That is, they satisfy the discrete orthogonality conditions
\begin{equation}
\sum_{u=1}^U r_i(\vt_u)r_j(\vt_u) =0\qquad\forall\;i\neq j.
\end{equation}
This implies that $\vR^\T\vR$ is diagonal. Hence, the condition number of $\vR$ equals the maximum ratio between the lengths of the columns of $\vR$, and there is a simpler relationship between the condition numbers of $\vLambda$ and $\vC$.
For instance, in the univariate case, suppose that the bases $\vp=(p_1,\ldots,p_{d+1})$ and $\vr=(r_1,\ldots,r_{2d+1})$ consist of Chebyshev polynomials of the first kind up to degrees $d$ and $2d$, respectively. As discussed previously in Section~\ref{sec:chebyshev-basis}, this representation of univariate SOS polynomials is appealing due to its favorable numerical properties. It is well-known that the Chebyshev polynomials $r_1,\ldots,r_{2d+1}$ are mutually orthogonal with respect to Chebyshev points of the first kind (defined in \eqref{eq:ChebPts-1stkind}): They satisfy \eqref{eq:disc-ortho} when $\cC_{1,2d}=\{\vt_1,\ldots,\vt_{2d+1}\}$. Furthermore, in this case one also has $\sum_{u=1}^{2d+1}r_i(\vt_u)^2=\rho_i$ where $\rho_1=2d+1$ and $\rho_i=d+1/2$ for $i>1$ (see, e.g., \cite[Sec.~4.6]{MasonHandscomb2003}).
Therefore, $\cond(\vR)=\sqrt{2}$ and $\cond(\vLambda)\leq\sqrt{2}\cond(\vC)$. This shows that, with an appropriate choice of interpolation points, the interpolant basis representation of univariate SOS polynomials is never much worse conditioned than their Chebyshev basis representation.
}

\section{Recovering Optimal SOS Decompositions}\label{sec:recovery}

Consider for simplicity the case where the problem \eqref{eq:CP-D} corresponds to an optimization problem over a single SOS cone ($\cK^*=\Sigma_{n,2d}$). Proposition~\ref{thm:alg-complex} indicates that Algorithm~\ref{alg:SY} terminates with a solution $\vz^*=(\bx^*,\vy^*,\bs^*)\in\cN(\eta)$ that satisfies the conditions \eqref{eq:alg-complex}, and if problems (\ref{eq:CP-P}-\ref{eq:CP-D}) are both feasible and have zero duality gap, then $\vx^*/\tau$ and $(\vy^*,\vs^*)/\tau$ are approximately optimal primal and dual solutions to (\ref{eq:CP-P}-\ref{eq:CP-D}), respectively. However, while the entries of $\vs^*\in\Sigma_{n,2d}$ are the coefficients of an SOS polynomial in the chosen basis $\vq$, an explicit SOS decomposition of this polynomial is not directly available from the output of Algorithm~\ref{alg:SY}. In this section, we describe how to construct a matrix $\vS^*\in\bS^L_{++}$ that satisfies $\Lambda^*(\vS^*)=\vs^*$ without recourse to semidefinite programming. An SOS decomposition for $\vs^*$ can then be obtained from the eigenvalue or Cholesky decomposition of $\vS^*$.

Recall that the Hessian of the logarithmic barrier function $-\ln(\det(\cdot))$ at the positive definite matrix $\vQ\in\bS^L_{++}$ is the linear operator $\vR\mapsto\vQ^{-1}\vR\vQ^{-1}$. This Hessian induces the local norm $\vR\mapsto\big\|\vQ^{-1/2}\vR\vQ^{-1/2}\big\|_F$ on the space $\bS^L$ (see Appendix~\ref{sec:LHSCB}). The self-concordance of the logarithmic barrier function implies that given $\vQ\in\bS^L_{++}$, every $\vS\in\bS^L$ that satisfies the inequality $\big\|\vQ^{-1/2}( \vS-\vQ)\vQ^{-1/2}\big\|_F<1$ is positive definite. Our next result makes use of this observation.


\begin{theorem}\label{thm:SOS-deco}
Let $\vx\in(\Sigma_{n,2d}^*)^\circ$ and $\vs\in\Sigma_{n,2d}^\circ$. Define the vector $\vw\defeq H(\vx)^{-1}\vs$ and the matrix $\vS\defeq\Lambda(\vx)^{-1}\Lambda(\vw)\Lambda(\vx)^{-1}$. Then $\Lambda^*(\vS) = \vs$. Furthermore, if \allowbreak $\big\|H(\vx)^{-1/2}(\vs+\delta g(\vx))\big\|<\delta$ for some $\delta>0$, then $\vS\succ\vzero$.
\end{theorem}

\begin{proof}
First note that $\Lambda(\vx)$ and $H(\vx)$ are both positive definite because $\vx\in(\Sigma_{n,2d}^*)^\circ$.
From (\ref{eq:logdetgrad}-\ref{eq:logdetHess}), recall that $\Lambda^*(\Lambda(\vx)^{-1})=-g(\vx)$ and $\Lambda^*\big(\Lambda(\vx)^{-1}\Lambda(\vw)\Lambda(\vx)^{-1}\big)=H(\vx)\vw$. Therefore, we have
\begin{equation*}
\Lambda^*(\vS)=\Lambda^*\left(\Lambda(\vx)^{-1}\Lambda(\vw)\Lambda(\vx)^{-1}\right)=H(\vx)\vw=\vs.
\end{equation*}
Now suppose $\big\|H(\vx)^{-1/2}(\vs+\delta g(\vx))\big\|<\delta$ for some $\delta>0$. We show that $\vS\succ\vzero$. From the definition of $\vS$, we have $\Lambda(\vx)\vS\Lambda(\vx)=\Lambda(\vw)$. Using this, we get
\begin{align*}
\left\|\Lambda(\vx)^{1/2}\left(\vS-\delta\Lambda(\vx)^{-1}\right)\Lambda(\vx)^{1/2}\right\|_F^2
&=\left(\Lambda(\vx)(\vS-\delta\Lambda(\vx)^{-1})\Lambda(\vx)\right)\bullet\left(\vS-\delta\Lambda(\vx)^{-1}\right)\\
&=\left(\Lambda(\vw)-\delta\Lambda(\vx)\right)\bullet\left(\Lambda(\vx)^{-1}(\Lambda(\vw)-\delta\Lambda(\vx))\Lambda(\vx)^{-1})\right)\\
&=(\vw-\delta\vx)^\T\Lambda^*\left(\Lambda(\vx)^{-1}(\Lambda(\vw)-\delta\Lambda(\vx))\Lambda(\vx)^{-1})\right)\\
&=(\vw-\delta\vx)^\T H(\vx)(\vw-\delta\vx)\\
&=(\vs+\delta g(\vx))^\T H(\vx)^{-1}(\vs+\delta g(\vx))\\
&=\left\|H(\vx)^{-1/2}(\vs+\delta g(\vx))\right\|^2<\delta^2.
\end{align*}
Finally, note that $\delta\Lambda(\vx)^{-1}\succ\vzero$ since $\Lambda(\vx)\succ\vzero$ and $\delta>0$.
Letting $\vQ=\delta\Lambda(\vx)^{-1}$, the discussion preceding this theorem now implies that $\vS\succ\vzero$ because $\big\|\delta^{-1/2}\Lambda(\vx)^{1/2}(\vS-\delta\Lambda(\vx)^{-1})\delta^{-1/2}\Lambda(\vx)^{1/2}\big\|_F< 1$.
\end{proof}

We will now show that the iterates of Algorithm \ref{alg:SY} satisfy the conditions of \mbox{Theorem \ref{thm:SOS-deco}}. Consider an iterate $\vz=(\bx,\vy,\bs)$ computed in the predictor or corrector phase. Given that $\vz\in\cN(\theta)$ for some $0\leq\theta<1$, it satisfies
\begin{equation}\label{eq:eta-nhd}
\begin{aligned}
&\left\|\bH(\bx)^{-1/2}(\bs+\mu(\vz)\bg(\bx))\right\|^2\\
&\hspace{2em}=\left\|H(\vx)^{-1/2}(\vs+\mu(\vz) g(\vx))\right\|^2+\tau^2(\kappa-\mu(\vz)/\tau)^2
\leq\theta^2\mu(\vz)^2.
\end{aligned}
\end{equation}
Now one can invoke Theorem~\ref{thm:SOS-deco} with $\delta=\mu(\vz)$ to construct a matrix $\vS\in\bS^L_{++}$ such that $\Lambda^*(\vS)=\vs$. To show that this assignment satisfies the requirements of Theorem~\ref{thm:SOS-deco}, we need to verify $\vx\in(\Sigma_{n,2d}^*)^\circ$, $\vs\in(\Sigma_{n,2d})^\circ$, $\mu(\vz)>0$, and $\big\|H(\vx)^{-1/2}(\vs+\mu(\vz) g(\vx))\big\|< \mu(\vz)$. The first three conditions are immediate; the last condition follows from \eqref{eq:eta-nhd} using $\mu(\vz)>0$ and $0\leq\theta<1$.

\begin{remark}
The matrix $\vS$ defined in Theorem~\ref{thm:SOS-deco} is the optimal solution to the equality-constrained least-squares problem
\begin{equation}\label{eq:SOS-deco-LS}
\begin{aligned}
&\minimize_{\vS\in\bS^L}\quad   && \left\|\Lambda(\vx)^{1/2}(\vS-\delta\Lambda(\vx)^{-1})\Lambda(\vx)^{1/2}\right\|_F^2\\
&\st\;\;                        && \Lambda^*(\vS) = \vs.
\end{aligned}
\end{equation}
\end{remark}

An alternative approach to computing a matrix $\vS\in\bS^L_{++}$ that satisfies $\Lambda^*(\vS) = \vs$ would be to minimize the distance $\|\vS-\delta\Lambda(\vx)^{-1}\|_F^2$ subject to $\Lambda^*(\vS) = \vs$. A similar least-squares problem was previously suggested in \cite{PeyrlParrilo2008,KaltofenLiYangZhi2008}. Its optimal solution $\vS^*$ is only guaranteed to be positive definite if $\|\vS^*-\delta\Lambda(\vx)^{-1}||_F<\delta\lambdamin(\Lambda(\vx)^{-1})$ (see \cite[Prop.~8]{PeyrlParrilo2008}). Instead, our least-squares problem \eqref{eq:SOS-deco-LS} minimizes a weighted Frobenius distance which has a natural interpretation in interior-point method theory. More importantly, the resulting solution is always guaranteed to be positive definite for solutions obtained from Algorithm~\ref{alg:SY}.

\section{Weighted Sum-of-Squares Polynomials}\label{sec:WSOS}

In this section, we turn our attention to optimization over WSOS cones and discuss how the results presented in earlier sections for SOS cones can be generalized to WSOS cones.
Recall that $\cV_{n,2\vd}^\vg$ is the space of polynomials $p$ for which there exist $s_1\in\cV_{n,2d_1},\ldots,s_m\in\cV_{n,2d_m}$ such that $p = \sum_{i=1}^m g_i s_i$. A polynomial $p\in\cV_{n,2\vd}^\vg$ belongs to $\Sigma_{n,2\vd}^\vg$ if there exist $s_1\in\Sigma_{n,2d_1},\ldots,s_m\in\Sigma_{n,2d_m}$ such that $p = \sum_{i=1}^m g_i s_i$. It is clear that $\Sigma_{n,2\vd}^\vg$ is a convex cone. Furthermore, $\Sigma_{n,2\vd}^\vg$ has the same dimension as $\cV_{n,2\vd}^\vg$ because each $\Sigma_{n,2d_i}$ is full-dimensional in $\cV_{n,2d_i}$. The next result characterizes when $\Sigma_{n,2\vd}^\vg$ is also pointed and closed, and thus a proper cone. We defer its proof to Appendix~\ref{sec:proofs}.

\begin{proposition}\label{thm:WSOSprops}
The cone $\Sigma_{n,2\vd}^\vg\subset\cV_{n,2\vd}^\vg$ is proper if and only if the following system is infeasible:
\begin{equation}\label{eq:WSOSprops}
\sum_{i=1}^m g_is_i=0,\qquad
\sum_{i=1}^m s_i\neq 0,\qquad
s_1\in\Sigma_{n,2d_1},\ldots,s_m\in\Sigma_{n,2d_m}.
\end{equation}
\end{proposition}

The system \eqref{eq:WSOSprops} must be infeasible when the set $\cS=\{\vt\in\R^n\,|\,g_i(\vt)\geq 0\quad\forall i=1,\ldots,m\}$ is unisolvent for $\cV_{n,2\vd}^\vg$. Therefore, Proposition~\ref{thm:WSOSprops} implies that $\Sigma_{n,2\vd}^\vg\subset\cV_{n,2\vd}^\vg$ is proper in this case. The conclusion that $\Sigma_{n,2\vd}^\vg$ must be closed when $\cS$ is unisolvent for $\cV_{n,2\vd}^\vg$ can also be deduced from \cite[Thm.~3.1]{Marshall2003}.

\deletethis{
\subsection{Barrier functions for the dual WSOS cone}
}

Let $U\defeq\dim\cV_{n,2\vd}^\vg$ and $L_i\defeq\dim\cV_{n,d_i}={n+d_i \choose n}$ for $i=1,\ldots,m$. The next theorem generalizes Proposition~\ref{thm:Nesterov2000} to the weighted case and shows that $\Sigma_{n,2\vd}^\vg$ is semidefinite representable.

\begin{proposition}\cite[Thm.~17.6]{Nesterov2000}\label{thm:Nesterov2000WSOS}
Fix an ordered basis $\vq = (q_1,\ldots,q_U)$ of $\cV_{n,2\vd}^\vg$ and an ordered basis $\vp_i=(p_{i,1},\ldots,p_{i,L_i})$ of $\cV_{n,d_i}$ for $i=1,\ldots,m$. Let $\Lambda_i:\R^U\to\bS^{L_i}$ be the unique linear mapping satisfying $\Lambda_i(\vq)=g_i\vp_i\vp_i^\T$, and let $\Lambda_i^*$ denote its adjoint. Then $\vs\in\Sigma_{n,2\vd}^\vg$ if and only if there exist matrices $\vS_1\succcurlyeq\vzero,\ldots,\vS_m\succcurlyeq\vzero$ satisfying
\begin{equation*}
\vs = \sum_{i=1}^m\Lambda_i^*(\vS_i).
\end{equation*}
Additionally, the dual cone of $\Sigma_{n,2\vd}^\vg$ admits the characterization
\begin{equation*}
\Sigma_{n,2\vd}^{\vg\;\;\;*} = \left\{\vx\in\real^U\,|\,\Lambda_i(\vx)\succcurlyeq\vzero\quad\forall\, i=1,\ldots,m\right\}.
\end{equation*}
\end{proposition}

Proposition~\ref{thm:Nesterov2000WSOS} describes the cone $\Sigma_{n,2\vd}^{\vg\;\;\;*}$ with $m$ linear matrix inequalities. Therefore, as before with the SOS cone, one can obtain an LHSCB for $\Sigma_{n,2\vd}^{\vg\;\;\;*}$ from restrictions of the logarithmic barrier function for the positive semidefinite cone. Let $F_i(\vx)\defeq -\ln(\det(\Lambda_i(\vx)))$ for $i=1,\ldots,m$. Then $F\defeq\sum_{i=1}^m F_i$ is an LHSCB for $\Sigma_{n,2\vd}^{\vg\;\;\;*}$ with barrier parameter at most $\sum_{i=1}^m L_i$ (see \cite[Thm.~2.3.1 and 2.3.9]{Renegar2001}). Furthermore, the gradient and Hessian of each $F_i$ admits characterizations analogous to (\ref{eq:logdetgrad}-\ref{eq:logdetHess}):
Let $\vE_{i1},\ldots,\vE_{iU}\in\bS^{L_i}$ be such that $\Lambda_i(\vx)=\sum_{u=1}^U\vE_{iu}x_u$. Then
\begin{align*}
\frac{\partial F_i}{\partial x_u}(\vx) &= -\Lambda_i(\vx)^{-1} \bullet \vE_{iu}, \qquad u=1,\ldots,U,\\
\frac{\partial^2 F_i}{\partial 	x_u \partial x_v}(\vx) &=
\left(\Lambda_i(\vx)^{-1}\vE_{iv}\Lambda_i(\vx)^{-1}\right)\bullet\vE_{iu}, \qquad u,v = 1, \ldots, U.
\end{align*}

The expressions for the gradient and Hessian of each $F_i$ can again be simplified when polynomials are represented via interpolation. For this, consider a set $\{\vt_1,\ldots,\vt_U\}\subset\R^n$ which is unisolvent for $\cV_{n,2\vd}^\vg$, and let $\vq$ be its corresponding Lagrange basis. Then each matrix $\vE_{iu}$ takes the form $\vE_{iu}=g_i(\vt_u)\vp_i(\vt_u)\vp_i(\vt_u)^\T$, and each operator $\Lambda_i$ becomes
\[ \Lambda_i(\vx) = \vP_i^\T\diag(\vg_i\circ\vx)\vP_i, \]
where $\vg_i\defeq (g_i(\vt_1),\ldots,g_i(\vt_U))$ and $\vP_i\defeq(p_{i,\ell}(\vt_u))_{u=1,\ldots,U;\ell=1,\ldots,L_i}$. Moreover, the gradient and Hessian of $F_i$ simplify into
\begin{align*}
\nabla F_i(\vx) &= -\vg_i\circ \diag\left(\vP_i(\vP_i^\T\diag(\vg_i\circ\vx)\vP_i)^{-1}\vP_i^\T\right),\\
\nabla^2 F_i(\vx) &= \left(\vg_i\vg_i^\T\right) \circ \left(\vP_i(\vP_i^\T\diag(\vg_i\circ\vx)\vP_i)^{-1}\vP_i^\T\right)^{\circ 2}.
\end{align*}
It is now possible to show, as in Theorem~\ref{thm:interpolant-basis}, that the gradient and Hessian of $F_i$ can be computed in $\Oh(L_iU^2)$ arithmetic operations using these formulas.

For an optimization problem over the cone $\Sigma_{n,2\vd}^\vg$, assuming for simplicity that all $d_i$'s have the same value $d$ and letting $L={n+d\choose n}$ and $U=\dim \cV_{n,2\vd}^\vg$, each iteration of Algorithm~\ref{alg:SY} requires $\Oh(U^3)$ arithmetic operations to compute the predictor and corrector directions and $\Oh(mLU^2)$ arithmetic operations to compute the Hessian of the barrier function $F$ in the interpolant basis representation. Therefore, each iteration of Algorithm~\ref{alg:SY} can be performed in $\Oh(mLU^2+U^3)$ arithmetic operations. \revise{In contrast, each iteration of a standard primal-dual interior-point method applied to the usual semidefinite programming formulation of the same problem requires $\Oh(mL^6)$ arithmetic operations (see Section \ref{sec:complexity}).}



For every iterate $\vz=(\bx,\vy,\bs)$ computed in Algorithm~\ref{alg:SY}, an explicit WSOS decomposition for the polynomial corresponding to $\vs$ can be recovered as in Section~\ref{sec:recovery}. For this, we need the following generalization of Theorem~\ref{thm:SOS-deco}. Its proof is similar to the proof of Theorem~\ref{thm:SOS-deco}.

\begin{theorem}\label{thm:WSOS-deco}
Let $\vx\in\big({\Sigma_{n,2\vd}^{\vg\;\;\;*}}\big)^\circ$ and $\vs\in\big(\Sigma_{n,2\vd}^\vg\big)^\circ$. Define the vector $\vw\defeq H(\vx)^{-1}\vs$ and the matrices $\vS_i\defeq\Lambda_i(\vx)^{-1}\Lambda_i(\vw)\Lambda_i(\vx)^{-1}$ for $i=1,\ldots,m$. Then $\sum_{i=1}^m \Lambda_i^*(\vS_i) = \vs$. Furthermore, if $\big\|H(\vx)^{-1/2}(\vs+\delta g(\vx))\big\|<\delta$ for some $\delta>0$, then $\vS_i\succ\vzero$ for $i=1,\ldots,m$.
\end{theorem}

Given an iterate $\vz=(\bx,\vy,\bs)$ of Algorithm~\ref{alg:SY}, positive definite matrices $\vS_1,\ldots,\vS_m$ satisfying $\sum_{i=1}^m \Lambda_i^*(\vS_i) = \vs$ can be computed using Theorem~\ref{thm:WSOS-deco}. A WSOS decomposition for the polynomial corresponding to $\vs$ can then be obtained from an eigenvalue or Cholesky decomposition of these matrices.

\section{Numerical Illustration}\label{sec:experiments}

This section reports the results of numerical experiments with a simple Matlab implementation of our approach to confirm the theoretical predictions of Sections~\ref{sec:barriers} and \ref{sec:stability}. Before we proceed to discuss these experiments, we discuss some details regarding our implementation of Algorithm~\ref{alg:SY}. Our Matlab code for this implementation and the numerical experiments below is available for download at \url{https://github.com/dpapp-github/alfonso}.

\subsection{Implementation details}

Our results in Section~\ref{sec:SY} and the analysis of the Skajaa--Ye algorithm in \cite{PappYildiz2017corrigendum} provide theoretically safe choices for the parameters of Algorithm~\ref{alg:SY}, which comprise the (fixed) predictor and corrector step lengths $\alpha_p$ and $\alpha_c$, the number $r_c$ of corrector steps to take in each corrector phase, and the sizes $\beta$ and $\eta$ of the neighborhoods where the iterates must remain at the end of the predictor and corrector phases respectively. The analysis of the algorithm remains valid if the fixed step length in the predictor phase is replaced with a line search. Specifically, starting from the ``safe'' fixed step length, we can search for the (approximately) largest step length $\alpha_p$ for which the iterate $\vz$ after the predictor phase remains in $\cN(\beta)$. Similarly, in the corrector phase, we need not always take $r_c$ steps; instead, we can check after each corrector step whether $\vz$ is already back in the neighborhood $\cN(\eta)$ and terminate the corrector phase if so. Both of these changes can improve the practical efficiency of the method without affecting its theoretical complexity; we included both of them in our implementation. For the remaining parameters, we used the values $\eta=0.0305$, $\beta=0.2387$, $\alpha_c=1$, $r_c=4$. Note that these parameters are generic values derived from the revised complexity analysis of Algorithm~\ref{alg:SY} in \cite{PappYildiz2017corrigendum}; they have not been tuned for the problems in this section or even for SOS optimization.

\subsection{Polynomial envelopes}\label{sec:envelope}

In this section we present numerical experiments which demonstrate the stability of our approach for SOS optimization problems with high-degree polynomials and discuss its practical advantages over solving the equivalent semidefinite programs.

\subsubsection{Problem description}

For our experiments in this section, we consider a family of optimization problems that was also studied in \cite{Papp2017} in the univariate case: given polynomials $f_1,\dots,f_k \in\cV_{n,\delta}$ and a set $\cS\subset\R^n$ defined as in \eqref{eq:domain},
find the polynomial $f\in\cV_{n,2d}$ that provides the closest lower approximation of $\min(f_1,\dots,f_k)$ on $\cS$, where the minimum is understood pointwise. Formally, we would like to compute the optimal solution to
\begin{equation}\label{eq:envelope}
\begin{aligned}
&\maximize_{f\in\cV_{n,2d}} &&\; \int_\cS f(\vt) d\vt\\
&\st\, &&\; f(\vt)\leq f_j(\vt)\;\; \forall\,\vt\in \cS\quad j=1,\dots,k.
\end{aligned}
\end{equation}
In this problem, the decision variable is the polynomial $f$, and the constraints require that the polynomial $f_j-f$ is nonnegative on $\cS$ for $j=1,\ldots,k$. Let $\deg(g_i)$ denote the degree of $g_i$. Assuming $\delta\leq 2d$, these constraints can be approximated with the requirements that the polynomial $f_j-f$ belongs to $\Sigma_{n,2\vd}^\vg$ for some $\vd=(d_1,\ldots,d_m)$ such that $2d_i+\deg(g_i)\leq 2d$ for $i=1,\ldots,m$.

Once a basis $\vq$ has been fixed for $\cV_{n,2d}$, the polynomial $f$ can be expressed as $f=\sum_{u=1}^U y_uq_u$ for some $\vy=(y_1,\ldots,y_U)$, and the objective function becomes $\int_\cS f(\vt) d\vt=\sum_{u=1}^U w_uy_u$ where $w_u\defeq\int_\cS q_u(\vt) d\vt$ for $u=1,\ldots,U$. Similarly, each polynomial $f_j$ can be expressed as $f_j=\sum_{u=1}^U y_{ju}q_u$ for some $\vy_j=(y_{j1},\ldots,y_{jU})$. In this notation, the SOS approximation to \eqref{eq:envelope} can be stated as
\begin{equation}\label{eq:envelope-SOS}
\begin{aligned}
&\maximize_{\vy\in\R^U} &&\; \sum_{u=1}^U w_uy_u\\
&\st\, &&\; \vy_j-\vy\in\Sigma_{n,2\vd}^\vg\quad j=1,\dots,k.
\end{aligned}
\end{equation}


\deletethis{
Assuming $\delta\leq 2d$, all polynomials involved can be represented as Lagrange interpolants on the same $U={n+2d\choose n}$ points. The decision variables are the function values $p(\vt_\ell)$ at the interpolation points $t_\ell$ for $\ell=1,\dots,U$. The nonnegativity of the polynomials $f_j(t)-f(t)$ can be approximated by these polynomials belonging to the cone $\Sigma_{n,2\vd}^\vg$.
}

In the examples below, $k=2$, $\delta=5$, and $\cS=[-1,1]^n$. The weights in the WSOS constraints are $g_j(\vt)=1-t_j^2$ for $j=1,\dots,n$ and $g_{n+1}(\vt)=1$, and the degrees are $d_j=d-1$ for $j=1,\ldots,n$ and $d_{n+1}=d$. Using Proposition~\ref{thm:WSOSprops}, it is easily verified that $\Sigma_{n,2\vd}^\vg$ (and hence its dual) is a proper cone. The exact semidefinite representation of $\Sigma_{n,2\vd}^\vg$ depends on the bases chosen for $\cV_{n,d_1},\ldots,\cV_{n,d_m}$, as described in Proposition~\ref{thm:Nesterov2000WSOS}. Below we experiment with the various choices discussed earlier in the paper.

\subsubsection{Results}\leavevmode

\textbf{Direct SOS optimization versus SDP in the interpolant basis.}
In \cite{Papp2017}, the problem \eqref{eq:envelope-SOS} was solved for $n=1$ and increasing values of $d$, using semidefinite programming and the interpolant basis representation described in Section~\ref{sec:interpolant-basis}. The basis $\vq$ was chosen as the Lagrange basis corresponding to Chebyshev points of the second kind (see \eqref{eq:ChebPts-2ndkind}).
It was found that even for $k=2$ and $\delta=5$, the largest instance for which the semidefinite programming formulation could be solved was approximately $2d=1100$ (the precise limit depending on the solver) before the solvers ran out of 32GB of memory. It was also reported that none of the tested solvers (SeDuMi \cite{sedumi}, SDPT3 \cite{sdpt3}, and CSDP \cite{csdp}) reported any numerical errors even for the highest degrees. In our first experiment, we compare this approach against optimizing directly over WSOS cones using Algorithm~\ref{alg:SY}.

\deletethis{
To this end, we first need to formulate the dual problem of \eqref{eq:envelope-SOS}:
\begin{equation}
\begin{split}
\minimize_\vx   &\quad \sum_{j,u} f_j(t_u)x_{ju}\\
\st             &\quad \sum_j x_{ju}=w_u\quad \forall\,u=1,\ldots,U,\\
                &\quad (x_{j1},\dots,x_{jU}) \in (\Sigma_{n,2\vd}^\vg)^* \quad j=1,\dots,k\\
\end{split}\label{eq:envelope-dual}
\end{equation}
with the weight vector $\vg=((1-t_1^2),\dots,(1-t_n^2),1)$ and degree vector $2\vd=(2d-2,2d-2,\dots,2d-2,2d)$.
Using Proposition~\ref{thm:WSOSprops} it is easily verified that $\Sigma_{n,2\vd}^\vg$ (and hence $(\Sigma_{n,2\vd}^\vg)^*$) is a proper convex cone. The gradient and Hessian of an LHSCB for the dual cone can be computed efficiently following Theorem~\ref{thm:interpolant-basis} and the discussion in Section~\ref{sec:WSOS}. The semidefinite programming formulation of \eqref{eq:envelope-dual} is standard, and can be derived from Proposition~\ref{thm:Nesterov2000WSOS}; we omit the details, which can also be found in \cite{Papp2017}.}

To solve the semidefinite programs, we used Mosek version 8.1.0.30 \cite{mosek} in addition to the solvers mentioned above. The performance of Mosek was at least as good as the performance of the other solvers in all instances; therefore, we report only the results obtained using Mosek here. All solvers were interfaced via Matlab R2016a. All computational results were obtained on a standard desktop computer equipped with 32GB RAM and a 4 GHz Intel Core i7 processor with 4 cores.

The results for $n=1$ are summarized in Table~\ref{tbl:wsos-vs-SDP-1D}. \deletethis{The primal infeasibility, dual infeasibility, and relative gap tolerances for both Mosek and Algorithm~\ref{alg:SY} were set to $10^{-8}$.} While the largest instance that we could solve with Mosek was the one with polynomials of degree $2d=1000$, we had no difficulty scaling our approach to $2d=10000$. Additionally, as expected, optimizing directly over WSOS cones is orders of magnitude faster than solving the corresponding semidefinite programs, even for the smaller instances. No numerical errors were reported by either solver, and high-accuracy solutions were returned by both: the relative primal and dual infeasibility and the relative duality and complementarity gaps stayed below $10^{-8}$ for our approach and below $10^{-7}$ for Mosek.

\begin{table}
	\centering
	\begin{tabular}{rrrrrrrrrr}
		\toprule
		\multirow{2}{*}{$d$} & \multirow{2}{*}{$U$} & \multicolumn{2}{c}{\# of variables} & \multicolumn{2}{c}{\# of iterations} & \multicolumn{2}{c}{time/iteration [s]} & \multicolumn{2}{c}{solver time [s]} \\
		        &                     & SOS & SDP & SOS & SDP & SOS & SDP & SOS & SDP\\
		\midrule
100 & 201 & 402 & 20402 & 51 & 9 & 0.02 & 0.92 & 1.23 & 8.32 \\
200 & 401 & 802 & 80802 & 60 & 11 & 0.08 & 14.97 & 4.73 & 164.70 \\
300 & 601 & 1202 & 181202 & 70 & 11 & 0.23 & 82.92 & 15.89 & 912.08 \\
400 & 801 & 1602 & 321602 & 72 & 12 & 0.49 & 257.40 & 35.22 & 3088.75 \\
500 & 1001 & 2002 & 502002 & 76 & 9 & 0.95 & 661.14 & 72.27 & 5950.27 \\
600 & 1201 & 2402 & 722402 & 78 &  & 1.34 &  & 104.83 &  \\
800 & 1601 & 3202 & 1283202 & 81 &  & 2.97 &  & 240.60 &  \\
1000 & 2001 & 4002 & 2004002 & 84 &  & 5.10 &  & 428.78 &  \\
1200 & 2401 & 4802 & 2884802 & 90 &  & 8.45 &  & 760.06 &  \\
1400 & 2801 & 5602 & 3925602 & 93 &  & 12.02 &  & 1117.46 &  \\
1600 & 3201 & 6402 & 5126402 & 107 &  & 17.89 &  & 1914.46 &  \\
1800 & 3601 & 7202 & 6487202 & 94 &  & 24.12 &  & 2267.36 &  \\
2000 & 4001 & 8002 & 8008002 & 107 &  & 32.53 &  & 3481.08 &  \\
2200 & 4401 & 8802 & 9688802 & 103 &  & 38.84 &  & 4000.06 &  \\
2400 & 4801 & 9602 & 11529602 & 105 &  & 48.80 &  & 5124.18 &  \\
2600 & 5201 & 10402 & 13530402 & 108 &  & 68.42 &  & 7389.41 &  \\
2800 & 5601 & 11202 & 15691202 & 125 &  & 101.11 &  & 12639.31 &  \\
3000 & 6001 & 12002 & 18012002 & 118 &  & 119.09 &  & 14052.45 &  \\
4000 & 8001 & 16002 & 32016002 & 141 &  & 217.27 &  & 30634.95 &  \\
5000 & 10001 & 20002 & 50020002 & 135 &  & 441.03 &  & 59538.97 &  \\
	\bottomrule\\
	\end{tabular}
	\caption{Solver statistics from our approach and from the SDP-based approach using the interpolant basis representation for problem \eqref{eq:envelope-SOS}. In the column titles, SOS denotes direct optimization over the WSOS cone using Algorithm~\ref{alg:SY}, whereas SDP denotes the solution of the equivalent semidefinite program using Mosek. Results are shown for instances with $n=1$, $k=2$, $\delta=5$, and increasing degrees $d$. Missing entries indicate that Mosek ran out of 32GB of memory while solving the problem.}
	\label{tbl:wsos-vs-SDP-1D}
\end{table}

\deletethis{
\begin{table}
	\centering
	\begin{tabular}{rrrrrrrrr}
		\toprule
		\multirow{2}{*}{$d$} & \multicolumn{2}{c}{\# of variables} & \multicolumn{2}{c}{\# of iters.} & \multicolumn{2}{c}{time/iter [s]} & \multicolumn{2}{c}{solver time [s]} \\
		        & SOS & SDP & SOS & SDP & SOS & SDP & SOS & SDP\\
		\midrule
100 & 402 & 20402 & 51 & 9 & 0.02 & 0.92 & 1.23 & 8.32 \\
200 & 802 & 80802 & 60 & 11 & 0.08 & 14.97 & 4.73 & 164.70 \\
300 & 1202 & 181202 & 70 & 11 & 0.23 & 82.92 & 15.89 & 912.08 \\
400 & 1602 & 321602 & 72 & 12 & 0.49 & 257.40 & 35.22 & 3088.75 \\
500 & 2002 & 502002 & 76 & 9 & 0.95 & 661.14 & 72.27 & 5950.27 \\
600 & 2402 & 722402 & 78 &  & 1.34 &  & 104.83 &  \\
800 & 3202 & 1283202 & 81 &  & 2.97 &  & 240.60 &  \\
1000 & 4002 & 2004002 & 84 &  & 5.10 &  & 428.78 &  \\
1200 & 4802 & 2884802 & 90 &  & 8.45 &  & 760.06 &  \\
1400 & 5602 & 3925602 & 93 &  & 12.02 &  & 1117.46 &  \\
1600 & 6402 & 5126402 & 107 &  & 17.89 &  & 1914.46 &  \\
1800 & 7202 & 6487202 & 94 &  & 24.12 &  & 2267.36 &  \\
2000 & 8002 & 8008002 & 107 &  & 32.53 &  & 3481.08 &  \\
2200 & 8802 & 9688802 & 103 &  & 38.84 &  & 4000.06 &  \\
2400 & 9602 & 11529602 & 105 &  & 48.80 &  & 5124.18 &  \\
2600 & 10402 & 13530402 & 108 &  & 68.42 &  & 7389.41 &  \\
2800 & 11202 & 15691202 & 125 &  & 101.11 &  & 12639.31 &  \\
3000 & 12002 & 18012002 & 118 &  & 119.09 &  & 14052.45 &  \\
4000 & 16002 & 32016002 & 141 &  & 217.27 &  & 30634.95 &  \\
5000 & 20002 & 50020002 & 135 &  & 441.03 &  & 59538.97 &  \\
	\bottomrule\\
	\end{tabular}
	\caption{Solver statistics from our approach and from the SDP-based approach using the interpolant basis representation for problem \eqref{eq:envelope-SOS}. In the column titles, SOS denotes direct optimization over the WSOS cone using Algorithm~\ref{alg:SY}, whereas SDP denotes the solution of the equivalent semidefinite program using Mosek. Results are shown for instances with $n=1$, $k=2$, $\delta=5$, and increasing degrees $d$. Missing entries indicate that Mosek ran out of 32GB of memory while solving the problem.}
	\label{tbl:wsos-vs-SDP-1D}
\end{table}
}

The results are qualitatively similar for $n=2$ using the interpolant basis representation corresponding to Padua points \eqref{eq:PaduaPts}, and for $n=3$ using the interpolant basis representation corresponding to approximate Fekete points; see Tables~\ref{tbl:wsos-vs-SDP-2D}--\ref{tbl:wsos-vs-SDP-3D}. While the semidefinite programming formulations could not be solved with the available memory for $2000$ or more monomials, our SOS approach scales to more than ten thousand monomials. \deletethis{One reason for this difference is the number of variables used in the two approaches. More importantly, while the matrices $\vE_{iu}=g_i(\vt_u)\vp_i(\vt_u)\vp_i(\vt_u)^\T$ appear in the linear equality constraints of the semidefinite programming reformulations (and hence in their Newton systems), they are only needed in the computation of the barrier gradient and Hessian in Algorithm~\ref{alg:SY}. Because these matrices are dense, this leads to a significant difference in the memory requirements of the two approaches.} We did encounter some stalling with our implementation for the larger instances in the bivariate case ($n=2$), but the relative primal and dual infeasibility and the relative duality and complementarity gaps of the returned solutions stayed below $10^{-7}$.

\deletethis{The primal infeasibility, dual infeasibility, and relative gap tolerances were set to $10^{-8}$ in the instances with $n=2$ and to $10^{-6}$ in the instances with $n=3$.}

\begin{table}
	\centering
	\begin{tabular}{rrrrrrrrrr}
		\toprule
		\multirow{2}{*}{$d$} & \multirow{2}{*}{$U$} & \multicolumn{2}{c}{\# of variables} & \multicolumn{2}{c}{\# of iterations} & \multicolumn{2}{c}{time/iteration [s]} & \multicolumn{2}{c}{solver time [s]} \\
		        &                     & SOS & SDP & SOS & SDP & SOS & SDP & SOS & SDP\\
		\midrule
10 & 231 & 462 & 10582 & 71 & 10 & 0.03 & 0.40 & 2.02 & 4.04 \\
15 & 496 & 992 & 47672 & 97 & 12 & 0.16 & 8.30 & 15.25 & 99.64 \\
20 & 861 & 1722 & 142212 & 118 & 12 & 0.70 & 87.73 & 82.67 & 1052.75 \\
25 & 1326 & 2652 & 335452 & 146 & 14 & 2.17 & 560.60 & 317.53 & 7848.33 \\
30 & 1891 & 3782 & 679892 & 150 &  & 4.93 &  & 740.22 &  \\
35 & 2556 & 5112 & 1239282 & 161 &  & 9.81 &  & 1578.62 &  \\
40 & 3321 & 6642 & 2088622 & 180 &  & 20.83 &  & 3748.92 &  \\
45 & 4186 & 8372 & 3314162 & 194 &  & 37.58 &  & 7289.79 &  \\
50 & 5151 & 10302 & 5013402 & 193 &  & 56.10 &  & 10828.25 &  \\
55 & 6216 & 12432 & 7295092 & 210 &  & 97.22 &  & 20416.14 &  \\
60 & 7381 & 14762 & 10279232 & 220 &  & 145.78 &  & 32070.97 &  \\
65 & 8646 & 17292 & 14097072 & 240 &  & 230.99 &  & 55436.59 &  \\
70 & 10011 & 20022 & 18891112 & 247 &  & 334.07 &  & 82515.00 &  \\
		\bottomrule\\
	\end{tabular}
	\caption{Solver statistics from our approach and from the SDP-based approach using the interpolant basis representation for problem \eqref{eq:envelope-SOS}. Results are shown for instances with $n=2$, $k=2$, $\delta=5$, and increasing degrees $d$. Missing entries indicate that Mosek ran out of 32GB of memory while solving the problem.}
	\label{tbl:wsos-vs-SDP-2D}
\end{table}

\begin{table}
	\centering
	\begin{tabular}{rrrrrrrrrr}
		\toprule
		\multirow{2}{*}{$d$} & \multirow{2}{*}{$U$} & \multicolumn{2}{c}{\# of variables} & \multicolumn{2}{c}{\# of iterations} & \multicolumn{2}{c}{time/iteration [s]} & \multicolumn{2}{c}{solver time [s]} \\
		        &                     & SOS & SDP & SOS & SDP & SOS & SDP & SOS & SDP\\
		\midrule
6 & 455 & 1365 & 24822 & 61 & 10 & 0.10 & 1.77 & 6.20 & 17.66 \\
8 & 969 & 2907 & 106425 & 72 & 10 & 0.71 & 36.42 & 51.46 & 364.18 \\
10 & 1771 & 5313 & 341913 & 94 & 9 & 3.51 & 495.86 & 330.38 & 4462.70 \\
12 & 2925 & 8775 & 909090 & 103 &  & 11.86 &  & 1221.48 &  \\
14 & 4495 & 13485 & 2108340 & 121 &  & 35.76 &  & 4326.45 &  \\
16 & 6545 & 19635 & 4409919 & 135 &  & 96.75 &  & 13061.05 &  \\
18 & 9139 & 27417 & 8508675 & 153 &  & 234.40 &  & 35863.41 &  \\
20 & 12341 & 37023 & 15386448 & 171 &  & 565.67 &  & 96729.43 &  \\
	\bottomrule\\
	\end{tabular}
	\caption{Solver statistics from our approach and from the SDP-based approach using the interpolant basis representation for problem \eqref{eq:envelope-SOS}. Results are shown for instances with $n=3$, $k=2$, $\delta=5$, and increasing degrees $d$. Missing entries indicate that Mosek ran out of 32GB of memory while solving the problem.}
	\label{tbl:wsos-vs-SDP-3D}
\end{table}

We emphasize that our approach is implemented naively in Matlab without optimizing the code for speed. In contrast, Mosek is an industry-grade implementation that uses advanced heuristic strategies for greater efficiency and stability.

\deletethis{In contrast, Mosek is a leading commercial optimizer that is implemented in C and that uses advanced heuristic strategies to exploit problem structure.}

\deletethis{Furthermore, a more sophisticated implementation of our approach is likely to require fewer iterations and roughly the same amount of time per iteration. Therefore, we believe that the time per iteration is a more relevant indicator of the potential speedup than the total time.}

\textbf{Direct SOS optimization versus SDP in the monomial and Chebyshev bases.}
The conventional representation of polynomials in SOS optimization using \mbox{Proposition~\ref{thm:Nesterov2000}} does not rely on interpolants; instead, polynomials are typically represented with their coefficients in the monomial basis. This is implemented by choosing the monomial basis for both bases $\vp$ and $\vq$ (Section~\ref{sec:monomial-basis}). For problems involving WSOS polynomials whose domains are rectangular boxes, one can also make a case for the Chebyshev basis representation (Section~\ref{sec:chebyshev-basis}). In this section we compare the efficiency and stability of these approaches to our method.

\begin{table}
	\centering
	\begin{tabular}{rrrrrrr}
	\toprule
	\multirow{2}{*}{$d$} & \multicolumn{2}{c}{\# of iterations} & \multicolumn{2}{c}{time/iteration [s]} & \multicolumn{2}{c}{solver time [s]}\\
	                     & SOS & SDP & SOS & SDP & SOS & SDP\\
	\midrule	
100 & 51 & 11 & 0.02 & 0.09 & 1.23 & 1.04 \\
200 & 60 & 11 & 0.08 & 0.68 & 4.73 & 7.47 \\
300 & 70 & 11 & 0.23 & 2.75 & 15.89 & 30.28 \\
400 & 72 & 13 & 0.49 & 8.04 & 35.22 & 104.49 \\
500 & 76 & 13 & 0.95 & 19.26 & 72.27 & 250.35 \\
600 & 78 & 13 & 1.34 & 39.00 & 104.83 & 506.98 \\
800 & 81 & 14 & 2.97 & 116.83 & 240.60 & 1635.61 \\
1000 & 84 & 14 & 5.10 & 280.51 & 428.78 & 3927.14 \\
1200 & 90 & 16 & 8.45 & 727.39 & 760.06 & 11638.27 \\
1400 & 93 & 18 & 12.02 & 1291.06 & 1117.46 & 23239.08 \\
1600 & 107 & 15 & 17.89 & 2330.95 & 1914.46 & 34964.20 \\
1800 & 94 & 17 & 24.12 & 3824.21 & 2267.36 & 65011.61 \\
2000 & 107 & 16 & 32.53 & 6084.72 & 3481.08 & 97355.58 \\
2200 & 103 & 17 & 38.84 & 9189.42 & 4000.06 & 156220.10 \\
2400 & 105 & 15 & 48.80 & 12371.34 & 5124.18 & 185570.11 \\
2600 & 108 & 14 & 68.42 & 18611.52 & 7389.41 & 260561.32 \\
2800 & 125 &  & 101.11 &  & 12639.31 &  \\
3000 & 118 &  & 119.09 &  & 14052.45 &  \\
4000 & 141 &  & 217.27 &  & 30634.95 &  \\
5000 & 135 &  & 441.03 &  & 59538.97 &  \\
	\bottomrule\\
	\end{tabular}
	\caption{Solver statistics from our approach and from the SDP-based approach using the Chebyshev basis representation for problem \eqref{eq:envelope-SOS}. Results are shown for instances with $n=1$, $k=2$, $\delta=5$, and increasing degrees $d$. Missing entries indicate that Mosek needed more than 300000 seconds to return a solution.}
	\label{tbl:interpolants-vs-SDP_std-1D}
\end{table}

\begin{table}
	\centering
	\begin{tabular}{rrrrrrr}
	\toprule
	\multirow{2}{*}{$d$} & \multicolumn{2}{c}{\# of iterations} & \multicolumn{2}{c}{time/iteration [s]} & \multicolumn{2}{c}{solver time [s]}\\
	                     & SOS & SDP & SOS & SDP & SOS & SDP\\
	\midrule	
10 & 71 & 17 & 0.03 & 0.05 & 2.02 & 0.91 \\
15 & 97 & 16 & 0.16 & 0.37 & 15.25 & 5.87 \\
20 & 118 & 19 & 0.70 & 2.27 & 82.67 & 43.19 \\
25 & 146 & 18 & 2.17 & 11.49 & 317.53 & 206.77 \\
30 & 150 & 18 & 4.93 & 46.16 & 740.22 & 830.92 \\
35 & 161 & 17 & 9.81 & 139.68 & 1578.62 & 2374.55 \\
40 & 180 & 18 & 20.83 & 400.51 & 3748.92 & 7209.18 \\
45 & 194 & 19 & 37.58 & 1067.34 & 7289.79 & 20279.43 \\
50 & 193 & 19 & 56.10 & 2563.38 & 10828.25 & 48704.21 \\
55 & 210 & 18 & 97.22 & 6329.91 & 20416.14 & 113938.43 \\
60 & 220 & 20 & 145.78 & 14450.37 & 32070.97 & 289007.38 \\
65 & 240 &  & 230.99 &  & 55436.59 &  \\
70 & 247 &  & 334.07 &  & 82515.00 &  \\
	\bottomrule\\
\end{tabular}
\caption{Solver statistics from our approach and from the SDP-based approach using the Chebyshev basis representation for problem \eqref{eq:envelope-SOS}. Results are shown for instances with $n=2$, $k=2$, $\delta=5$, and increasing degrees $d$. Missing entries indicate that Mosek needed more than 300000 seconds to return a solution.}
    \label{tbl:interpolants-vs-SDP_std-2D}
\end{table}

\begin{table}
	\centering
	\begin{tabular}{rrrrrrr}
	\toprule
	\multirow{2}{*}{$d$} & \multicolumn{2}{c}{\# of iterations} & \multicolumn{2}{c}{time/iteration [s]} & \multicolumn{2}{c}{solver time [s]}\\
	                     & SOS & SDP & SOS & SDP & SOS & SDP\\
	\midrule
6 & 61 & 10 & 0.10 & 0.14 & 6.20 & 1.38 \\
8 & 72 & 12 & 0.71 & 0.89 & 51.46 & 10.70 \\
10 & 94 & 14 & 3.51 & 6.35 & 330.38 & 88.94 \\
12 & 103 & 14 & 11.86 & 42.79 & 1221.48 & 599.10 \\
14 & 121 & 15 & 35.76 & 221.75 & 4326.45 & 3326.32 \\
16 & 135 & 15 & 96.75 & 926.23 & 13061.05 & 13893.51 \\
18 & 153 & 16 & 234.40 & 3293.34 & 35863.41 & 52693.48 \\
20 & 171 & 15 & 565.67 & 13328.76 & 96729.43 & 199931.41 \\
	\bottomrule\\
\end{tabular}
\caption{Solver statistics from our approach and from the SDP-based approach using the Chebyshev basis representation for problem \eqref{eq:envelope-SOS}. Results are shown for instances with $n=3$, $k=2$, $\delta=5$, and increasing degrees $d$.}
	\label{tbl:interpolants-vs-SDP_std-3D}
\end{table}

Tables~\ref{tbl:interpolants-vs-SDP_std-1D}--\ref{tbl:interpolants-vs-SDP_std-3D} compare our approach against the semidefinite programming formulation in the Chebyshev basis for $n=1,2,3$ dimensions. The results reveal several interesting conclusions.
As expected, the Chebyshev basis outperforms the interpolant basis in terms of solution times when both bases are used in the semidefinite programming formulation; this is due to the fact that the solvers can exploit the sparsity of the matrices $\vE_{iu}$ in the Chebyshev basis representation. This sparsity also allows the SDP-based approach to scale to higher degrees with the Chebyshev basis representation than it does with interpolants. Regardless, our approach is more efficient than semidefinite programming even when the latter is used with the Chebyshev basis, and the improvement in running times increases with increasing degrees. In our largest experiment, our SOS approach achieved an over 30-fold speedup over the SDP-based approach using the Chebyshev basis representation.

We also investigated the performance of the solvers SeDuMi, SDPT3, and CSDP on the semidefinite programming formulation in the Chebyshev basis. While the performances of all three solvers were similar, they required longer solution times than Mosek. Furthermore, SeDuMi and SDPT3 reported numerical problems in some of the high-degree instances. The relative infeasibility of the returned solutions were above $10^{-7}$ for SeDuMi and $10^{-6}$ for SDPT3. 

The experiments using the standard semidefinite programming formulation in the monomial basis were unsuccessful: although Mosek did return an ``optimal solution'' of the semidefinite programs and reported no numerical errors, the returned solutions had relative complementarity gaps above $10^{-4}$ in all instances, even in the univariate case. Additionally, the computed ``optimal'' objective function values did not exhibit the expected monotonicity with respect to $d$, and they did not seem to converge to the true integral of $\min(f_1,f_2)$ on $\cS=[-1,1]^n$. \deletethis{The optimal values were also markedly different from the optimal values obtained using the Chebyshev and interpolant basis representations for $d\geq 10$.} Hence, the accuracy of these solutions is far below the accuracy of the solutions obtained in our other experiments. Moreover, in all instances Mosek needed about twice as many iterations with the monomial basis representation than with the Chebyshev basis representation; this was also likely a consequence of the poor conditioning of the monomial basis representation.

We repeated these experiments with the monomial basis representation using the semidefinite programming solvers SeDuMi, SDPT3, and CSDP; the results were qualitatively similar, indicating that the problem is indeed the conditioning of the formulation, and not the stability of the solvers. All three solvers reported numerical problems in most of the instances, and only returned inaccurate solutions after at least as many iterations as they needed to compute accurate solutions using the Chebyshev basis. This verifies and expands on similar results reported in \cite{Papp2017}.

In summary, these experiments confirm that the monomial basis is not suitable for the representation of polynomials in these problems; both the Chebyshev and interpolant bases are clearly superior choices. Additionally, although the Chebyshev basis is a better choice than the interpolant basis in the SDP-based approach, optimizing directly over WSOS cones (using the interpolant basis representation) is superior to semidefinite programming even when the semidefinite programs are formulated using the Chebyshev basis representation and solved with a solver that exploits their sparsity.

\subsection{Polynomial optimization}\label{sec:poly-opt}

In some applications of SOS optimization, one is interested in finding rational or algebraic solutions whose feasibility can be rigorously certified \cite{BachocVallentin2008,BallingerBlekhermanCohnGiansiracusaKellySchurmann2009,Vallentin2016}. While floating-point implementations can compute numerical solutions up to any accuracy, these solutions are never exact and cannot be used directly as certificates in these applications \cite{PeyrlParrilo2008,KaltofenLiYangZhi2008}. In this section, we show that our floating-point implementation of Algorithm~\ref{alg:SY} produces high-accuracy solutions which can be used to obtain rigorous rational certificates of tight global lower bounds for polynomial minimization problems. A \emph{Mathematica} notebook computing and verifying in rational arithmetic WSOS certificates using Theorem \ref{thm:WSOS-deco} is available as a supplementary material from \url{https://github.com/dpapp-github/certificates}.

\subsubsection{Problem description}

In this section we consider the polynomial minimization problem \eqref{eq:poly-opt}.
In other words, we would like to compute the largest $y\in\R$ such that $f(\vt)-y\geq 0$ for all $\vt\in\cS$.
This constraint can be approximated with the condition that $f(\vt)-y$ belongs to $\Sigma_{n,2\vd}^\vg$ for some $\vd=(d_1,\ldots,d_m)$ such that $2d_i+\deg(g_i)\geq\deg(f)$ for $i=1,\ldots,m$.
Choosing $\vq$ as the Lagrange basis corresponding to a set $\{\vt_1,\ldots,\vt_U\}\in\R^n$ which is unisolvent for $\cV_{n,2d}$,  the SOS approximation to \eqref{eq:poly-opt} becomes
\begin{equation}\label{eq:poly-opt-SOS}
\begin{aligned}
&\maximize_{y\in\R} &&\; y\\
&\st\, &&\; \left(f(\vt_u)-y\right)_{u=1,\ldots,U}\in\Sigma_{n,2\vd}^\vg.
\end{aligned}
\end{equation}
Its dual problem is
\begin{equation}\label{eq:poly-opt-SOS-dual}
\begin{aligned}
&\minimize_{\vx\in\R^U} &&\; \sum_{u=1}^U f(\vt_u)x_u\\
&\st\, &&\; \vone^\T\vx=1,\\
&&&\; \vx\in\Sigma_{n,2\vd}^{\vg\;\;\;*}.
\end{aligned}
\end{equation}

For our experiments, we use three test problems that are frequently used in the literature (see \cite{RayNataraj2009} and the references therein). In each example, $\cS=\prod_{j=1}^n [\ell_j,u_j]$.

\textit{Example 1.} Minimize Butcher's polynomial
$f(t_1,t_2,t_3,t_4,t_5,t_6)= t_6t_2^2 + t_5t_3^2 - t_1t_4^2 + t_4^3 + t_4^2 - 1/3 t_1 + 4/3 t_4$ over the hyper-rectangle $\cS=[-1,0]\times[-0.1,0.9]\times[-0.1,0.5]\times[-1,-0.1]\times[-0.1,-0.05]\times[-0.1,-0.03]$.

\textit{Example 2.} Minimize Caprasse's polynomial
$f(t_1,t_2,t_3,t_4)= -t_1t_3^3 + 4t_2t_3^2t_4 + 4t_1t_3t_4^2 + 2t_2t_4^3 + 4t_1t_3 + 4t_3^2 - 10t_2t_4 - 10t_4^2 + 2$ over the hyper-rectangle $\cS=[-0.5,0.5]^4$.

\textit{Example 3.} Minimize the 7-variable \deletethis{animal} magnetism polynomial
$f(t_1,t_2,t_3,t_4,t_5,t_6,t_7)= t_1^2 + 2t_2^2 + 2t_3^2 + 2t_4^2 + 2t_5^2 + 2t_6^2 + 2t_7^2 - t_1$ over the hyper-rectangle $\cS=[-1,1]^7$.

\deletethis{
Caprasse's polynomial
$f(t_1,t_2,t_3,t_4)= -t_1t_3^3 + 4t_2t_3^2t_4 + 4t_1t_3t_4^2 + 2t_2t_4^3 + 4t_1t_3 + 4t_3^2 - 10t_2t_4 - 10t_4^2 + 2$,
Butcher's polynomial
$f(t_1,t_2,t_3,t_4,t_5,t_6)= t_6t_2^2 + t_5t_3^2 - t_1t_4^2 + t_4^3 + t_4^2 - 1/3 t_1 + 4/3 t_4$,
and the 7-variable magnetism polynomial
$f(t_1,t_2,t_3,t_4,t_5,t_6,t_7)= t_1^2 + 2t_2^2 + 2t_3^2 + 2t_4^2 + 2t_5^2 + 2t_6^2 + 2t_7^2 - t_1$.
Each polynomial is associated with a domain $\cS=\prod_{j=1}^n [\ell_j,u_j]$. 
The standard choices are $\vell=(-0.5, -0.5, -0.5, -0.5)$ and $\vu=(0.5, 0.5, 0.5, 0.5)$ for Caprasse's polynomial, $\vell=(-1, -0.1, -0.1, -1, -0.1, -0.1)$ and $\vu=(0, 0.9, 0.5, -0.1, -0.05, -0.03)$ for Butcher's polynomial, and $\vell=(-1, -1, -1, -1, -1, -1, -1)$ and $\vu=(1, 1, 1, 1, 1, 1, 1)$ for the magnetism polynomial.
}
We use the weights $g_j(\vt)=(u_j-t_j)(t_j-\ell_j)$ for $j=1,\dots,n$, and $g_{n+1}(\vt)=1$ in the WSOS constraints. The corresponding degrees are $d_j=\lceil\deg(f)/2\rceil-1$ for $j=1,\dots,n$ and $d_{n+1}=\lceil\deg(f)/2\rceil$.

\subsubsection{Results}

The solution $\vt^*=(0, 0.9, 0.5, 0, -0.1, -0.1)$ is a global minimizer of Butcher's polynomial on its standard domain and achieves the optimal value $\text{OPT}=-2159/1500$. In this section, we describe how our floating-point implementation of Algorithm~\ref{alg:SY} can be used with additional post-processing in exact arithmetic to certify the global lower bound $\text{LB}=\text{OPT}-10^{-18}$. For these experiments, we use rational interpolation points, and we choose each basis $\vp_i$ as the basis of Chebyshev polynomials of the first kind up to degree $d_i$ for $i=1,\dots,n+1$.

To compute global lower bound certificates, we first solve \eqref{eq:poly-opt-SOS} in floating-point arithmetic. We then round the resulting approximate solution for \eqref{eq:poly-opt-SOS-dual} into a rational vector $\vx$. We also let $\vs=\left(f(\vt_u)-\text{LB}\right)_{u=1,\ldots,U}$. Note that $\vs$ is rational given that $\text{LB}$ and the interpolation points are rational. We then use Theorem~\ref{thm:WSOS-deco} to compute in exact arithmetic the matrices $\vS_1,\dots,\vS_{n+1}$ corresponding to the weights $g_1,\dots,g_{n+1}$; these matrices are guaranteed to be rational given that $\vx$, $\vs$, and the interpolation points are rational. Now one can check in exact arithmetic that $\vS_1,\dots,\vS_{n+1}$ satisfy
\[f-\text{LB}=\sum_{i=1}^{n+1} g_i\vp_i^\T\vS_i\vp_i\]
and verify their positive semidefiniteness via LDL factorization.


The same procedure can be used to certify global lower bounds which are $10^{-18}$ less than the true optimal value for the magnetism polynomial and $10^{-13}$ less than the true optimal value for Caprasse's polynomial.

\section{Conclusions}\label{sec:conc}

Several approaches have been proposed to mitigate the computational issues associated with the semidefinite programming representation of SOS polynomials. These include exploiting sparsity \cite{sparsepop} or symmetry \cite{GatermannParrilo2004}, and replacing the semidefinite programming hierarchies with linear and second-order cone programming hierarchies \cite{AhmadiMajumdar2014,Kuangetal2017}. Our approach of combining non-symmetric conic optimization algorithms and polynomial interpolants also appears to be a very promising and competitive alternative to the conventional SDP-based approach in terms of stability and efficiency. Furthermore, these improvements are achieved without resorting to approximations of the SOS cone or assumptions of sparsity or symmetry. We emphasize that our approach can also be used in conjunction with the techniques that exploit sparsity and symmetry in SOS optimization.

Our approach is particularly suited for problems that require polynomials of high degree. The use of high-degree polynomials is especially relevant in problems involving polynomial or rational function approximations of non-polynomial functions and in data-driven optimization, where the interpolant basis representation is the most natural representation. In particular, with our proposed approach, optimization models involving arbitrary smooth functions (which can be uniformly approximated with polynomials up to any accuracy) can be solved approximately without the explicit construction of the approximating high-degree polynomials.

While the good conditioning of the interpolant basis representation in SOS optimization using semidefinite programming was established in \cite{Papp2017}, and numerically confirmed at least in the univariate case, it has been unclear whether it is possible to optimize efficiently over the cone of SOS interpolants and its dual, circumventing the dense semidefinite programs used in the earlier paper.

The primary results of this work are that the interpolant basis representation also allows for simple and efficient computation of the gradient and Hessian of the logarithmic barrier function of the dual SOS cone in the multivariate and weighted cases as well as in the univariate case, and that with this tractable barrier function, we can solve SOS optimization problems without need for semidefinite programming formulations. The optimal solution of the semidefinite program and the associated SOS decompositions can still be recovered from the optimal solution of the SOS optimization problem with little additional effort.

The numerical results indicate that with Chebyshev and Padua points in the univariate and bivariate cases, and with approximate Fekete points in the multivariate case, the numerical performance of the proposed approach matches the theoretical predictions. Our approach is increasingly favorable to the conventional SDP-based approach as the degree increases; moreover, the asymptotic speedup is also an increasing function of the number of arguments of the polynomials.

\section*{Acknowledgments}

The authors would like to thank the Associate Editor and the Referees for their constructive feedback. 
We also thank Madhu Kiran Chowdary Kolli for his helpful comments on the presentation of the material.

\bibliographystyle{siamplain}
\bibliography{interpolants_siamart}

\bigskip
\appendix

\section{Review of LHSCBs}\label{sec:LHSCB}

In this appendix, we provide a brief review of notions that are central to interior-point method theory. Our presentation is based on the textbook \cite{Renegar2001}.

Let $\cK\subset\R^n$ be a proper cone, and let $\cK^\circ$ denote its interior. Throughout this appendix, we consider a twice continuously differentiable function $F:\cK^\circ\to\R$. Let $g$ and $H$ denote the gradient and Hessian of $F$. We assume that $H(\vx)$ is positive definite for all $\vx\in\cK^\circ$. For any $\vx\in\cK^\circ$, the \emph{local norm} at $\vx$ is defined as $\vv\mapsto\|H(\vx)^{1/2}\vv\|$. Let $\cB_\vx(\vu,r)\defeq\{\vv\in\R^n\,|\;\|H(\vx)^{1/2}(\vv-\vu)\|<r\}$ denote the open ball of radius $r>0$ centered at $\vu\in\R^n$ with respect to the local norm at $\vx$. The function $F$ is said to be \emph{self-concordant} if for all $\vx\in\cK^\circ$, one has $\cB_\vx(\vx,1)\subset\cK^\circ$, and for all $\vv\neq 0$ and $\vu\in\cB_\vx(\vx,1)$, one has
\begin{equation*}
1-\|H(\vx)^{1/2}(\vu-\vx)\|\leq\frac{\|H(\vu)^{1/2}\vv\|}{\|H(\vx)^{1/2}\vv\|}\leq\frac{1}{1-\|H(\vx)^{1/2}(\vu-\vx)\|}.
\end{equation*}
The function $F$ is said to be a \emph{logarithmically homogeneous self-concordant barrier} (LHSCB) if it is self-concordant and satisfies the following conditions:
\begin{enumerate}
\item[i.] $\nu\defeq\sup_{\vx\in\cK^\circ} \|H(\vx)^{-1/2}g(\vx)\|^2$ is finite, and
\item[ii.] $F(t\vx)=F(\vx)-\nu\ln t$ for all $\vx\in \cK^\circ$ and $t>0$.
\end{enumerate}
The quantity $\nu$ is called the \emph{barrier parameter} of $F$.

\bigskip
\section{Omitted Proofs}\label{sec:proofs}

\begin{proof}[Proof of Proposition~\ref{thm:WSOSprops}]
The cone $\Sigma_{n,2\vd}^\vg\subset\cV_{n,2\vd}^\vg$ is always convex and has nonempty interior. It can be shown using Corollary~9.1.3 in \cite{Rockafellar1997} that $\Sigma_{n,2\vd}^\vg$ is also closed whenever the system~\eqref{eq:WSOSprops} is infeasible. In the remainder of the proof, we show that $\Sigma_{n,2\vd}^\vg$ is pointed if and only if the system~\eqref{eq:WSOSprops} is infeasible. Let $d_{\max}=\max_{i=1,\ldots,m}d_i$. Suppose $\Sigma_{n,2\vd}^\vg$ is not pointed. Then there exists a nonzero polynomial $\ell$ such that $\ell,-\ell\in\Sigma_{n,2\vd}^\vg$. Let $s_i^+,s_i^-\in\Sigma_{n,2d_i}$ be such that $\ell=\sum_{i=1}^m g_is_i^+$ and $-\ell=\sum_{i=1}^m g_is_i^-$. Let $s_i\defeq s_i^+ + s_i^-$ for $i=1,\ldots,m$. It is clear that $\sum_{i=1}^m g_is_i=0$ and $s_i\in\Sigma_{n,2d_i}$ for $i=1,\ldots,m$. Furthermore, we cannot have $s_i^+=s_i^-=0$ for all $i=1,\ldots,m$ because $\ell$ is nonzero. Then $\sum_{i=1}^m s_i=\sum_{i=1}^m (s_i^+ + s_i^-)\neq 0$ because $\Sigma_{n,2d_{\max}}$ is pointed. For the converse, suppose there exist polynomials $s_1\in\Sigma_{n,2d_1},\ldots,s_m\in\Sigma_{n,2d_m}$ such that $\sum_{i=1}^m g_is_i=0$ and $\sum_{i=1}^m s_i\neq 0$. Assume without loss of generality that $s_1$ is nonzero. Then $g_1s_1$ is nonzero and belongs to $\Sigma_{n,2\vd}^\vg$. Furthermore, $-g_1s_1=\sum_{i=2}^m g_is_i\in\Sigma_{n,2\vd}^\vg$. This shows that $g_1s_1$ belongs to the lineality space of $\Sigma_{n,2\vd}^\vg$, and hence $\Sigma_{n,2\vd}^\vg$ is not pointed.
\end{proof}

\end{document}